\documentclass[reqno]{amsart}

\usepackage{amssymb,latexsym}
\usepackage{hyperref,setspace}

\usepackage[abbrev,nobysame]{amsrefs}
\newtheorem{theorem}{Theorem}

\newtheorem{definition}[theorem]{Definition}
\newtheorem{proposition}[theorem]{Proposition}
\newtheorem{corollary}[theorem]{Corollary}
\newtheorem{lemma}[theorem]{Lemma}
\theoremstyle{remark}
\newtheorem{remark}[theorem]{Remark}

\numberwithin{theorem}{section}
\numberwithin{equation}{section}
\makeatletter
\def\section{\@startsection{section}{1}%
  \z@{.7\linespacing\@plus\linespacing}{.5\linespacing}%
  {\small\scshape\centering}}

\renewcommand{\subsubsection}{\@startsection{subsubsection}{3}{\z@}%
     {.2\linespacing\@plus\linespacing}{.2\linespacing}
    {\reset@font\normalsize\bf}}
\makeatother

\DeclareMathOperator{\supp}{supp}
\begin{document}
\title[Maximizers for Wave Strichartz Inequalities]{Maximizers for the Strichartz Inequalities for The Wave Equation}
\author{Aynur Bulut}
\address{University of Texas at Austin, Department of Mathematics\\1 University Station C1200, Austin, TX 78712}
\email{abulut@math.utexas.edu}
\thanks{\today}
\begin{abstract}
We prove the existence of maximizers for Strichartz inequalities for the wave equation in dimensions $d\geq 3$.  Our approach follows the scheme given by Shao in \cite{ShaoSchrodinger} which obtains the existence of maximizers in the context of the Schr\"odinger equation. The main tool that we use is the linear profile decomposition for the wave equation which we prove in $\mathbb{R}^d$, $d\geq 3$, extending the profile decomposition result of Bahouri and Gerard \cite{BahouriGerard}, previously obtained in $\mathbb{R}^3$.

\end{abstract}
\maketitle
\allowdisplaybreaks
\section{INTRODUCTION}
We consider the initial value problem for the wave equation:
\begin{equation}
\left\lbrace\begin{split}
\label{lab1}
\partial_{tt}u-\Delta u&=0,\\
\hfill u(0)&=u_0\in \dot{H}^1(\mathbb{R}^d),\\
\hfill \partial_t u(0)&=u_1\in L^2(\mathbb{R}^d),
\end{split}\right.
\end{equation}
where $u(t,x)$ is a complex valued function on $\mathbb{R}\times\mathbb{R}^d$, $d\geq 3$.

The Strichartz inequalities associated to ($\ref{lab1}$) state that for a suitably chosen pair $(q,r)$ there exists a constant $W_{q,r}>0$ such that \begin{align}\label{lab2}\lVert u(t,x)\rVert_{L_t^qL_x^r(\mathbb{R}\times\mathbb{R}^d)}\leq W_{q,r}\lVert (u_0, u_1)\rVert_{\dot{H}^1\times L^2(\mathbb{R}^d)}\end{align}
whenever $u(t,x)$ solves ($\ref{lab1}$).

Estimates similar to the above
Strichartz inequalities were first introduced by Segal \cite{Segal}, and were further
studied by Strichartz \cite{Strichartz}.  Subsequently, a substantial
literature has developed to examine the most general forms of these
estimates. (See, e.g. the work of Ginibre and Velo \cite{GinibreVelo},
Keel and Tao \cite{KeelTao}).

The recent study of sharp constants and existence of maximizers for the
Strichartz inequalities began with the study of the
Schr\"odinger equation,\begin{align}i\partial_t u+\Delta u=0.\label{lab3}\end{align}  For this equation, much of the work has focused on
dimensions $1$ and $2$.  In dimension $1$, Kunze \cite{Kunze}, by using an application of concentration compactness technique, proved that the Strichartz inequality associated to ($\ref{lab3}$) has a maximizer.
Subsequently, Foschi \cite{Foschi}, working in dimensions $1$ and $2$, explicitly determined the sharp constants, and
characterized the maximizers for the Strichartz inequality for both the
Schr\"odinger and wave equations. Also in dimensions $1$ and $2$, but
independent of Foschi's work, Hundertmark and Zharnitsky \cite{HundertmarkZharnitsky} showed the
same characterization of maximizers for the Schr\"odinger equation by obtaining a new representation of the Strichartz integral. Related to this literature, following the arguments in \cite{HundertmarkZharnitsky}, Carneiro \cite{Carneiro} obtained a sharp inequality for the Strichartz norm, which admits only Gaussian maximizers and which obtains the sharp forms of the classical Strichartz inequality results in \cite{HundertmarkZharnitsky} and \cite{Foschi} as corollaries. Also recently, Bennett, Bez, Carbery and Hundertmark \cite{BennettBezCarberyHundertmark} proved a monotonicity of formula for the classical Strichartz norm as the initial data evolves under a certain quadratic heat-flow in $d=1,2$.

Another approach to prove the existence of maximizers is based on the profile decomposition.  This idea was employed by Shao in \cite{ShaoSchrodinger}, where he used linear profile decomposition results, given by B\'egout and Vargas in \cite{BegoutVargas},
to establish the existence of maximizers for the Schr\"odinger equation.  In \cite{ShaoAiry} he also established the linear profile decomposition for the Airy equation and obtained a dichotomy result on the existence of maximizers for the symmetric Airy-Strichartz inequality.

 The profile decomposition was introduced by Bahouri and Gerard \cite{BahouriGerard} in the context of the wave equation, in $\mathbb{R}^3$. The idea of this decomposition is connected to the concentration compactness method of P.L. Lions and the bubble decomposition for elliptic equations (see \cite{BrezisCoron} and \cite{Struwe}).  For the Schr\"odinger equation, the linear profile decomposition was independently proved for $d=2$ by Merle and Vega \cite{MerleVega}.  Carles and Keraani \cite{CarlesKeraani} treated the one-dimensional case, while in higher dimensions the result was obtained by Begout and Vargas in \cite{BegoutVargas}.  Recently, the profile decomposition has been used to prove a number of remarkable results.  In particular, it has been used as a key tool in establishing the concentration compactness/rigidity method introduced by Kenig and Merle in \cite{KenigMerleNLS} and \cite{KenigMerleWave}.  It was also used in the works of Tao, Visan and Zhang \cite{TaoVisanZhang} and  Killip, Tao and Visan \cite{KillipTaoVisan} as one of the main ingredients in proving the existence of minimal kinetic energy blowup solutions.

In this paper, inspired by the work of Shao in \cite{ShaoSchrodinger}, we establish the existence of maximizers for the Strichartz inequalities ($\ref{lab2}$) in dimensions $d\geq 3$.  The main tool that we use is the linear profile decomposition, which was previously obtained for $d=3$ by Bahouri and Gerard in \cite{BahouriGerard}.  Here, we extend this decomposition to dimensions $d\geq 3$ based on the arguments given in \cite{BahouriGerard}, \cite{Keraani} and \cite{Gerard}.  In particular, we prove the following version (for an analogous statement see Lemma $4.3$ in Kenig and Merle \cite{KenigMerleWave}):

\begin{theorem}
\label{lab4}
Let $(u_{0,n},u_{1,n})_{n\in\mathbb{N}}$ be a bounded sequence in $\dot{H}^1\times L^2(\mathbb{R}^d)$ with $d\geq 3$.  Then there exists a subsequence of $(u_{0,n},u_{1,n})$ (still denoted $(u_{0,n},u_{1,n})$), a sequence $(V_0^j,V_1^j)_{j\in\mathbb{N}}\subset \dot{H}^1\times L^2(\mathbb{R}^d)$ and a sequence of triples $(\epsilon_n^j,x_n^j,t_n^j)\in \mathbb{R^+}\times\mathbb{R}^d\times\mathbb{R}$ such that for every $j\neq j'$,
\begin{align}
\label{lab5}\frac{\epsilon_n^j}{\epsilon_n^{j'}}+\frac{\epsilon_n^{j'}}{\epsilon_n^j}+\frac{|t_n^j-t_n^{j'}|}{\epsilon_n^j}+\frac{|x_n^j-x_n^{j'}|}{\epsilon_n^j}\mathop{\longrightarrow}_{n\rightarrow\infty}
\infty,
\end{align}
and for every $l\geq 1$, if $V^j=S(t)(V_0^j,V_1^j)$ and $V_n^j(t,x)=\frac{1}{(\epsilon_n^j)^\frac{d-2}{2}}V^j\left(\frac{t-t_n^j}{\epsilon_n^j},\frac{x-x_n^j}{\epsilon_n^j}\right)$,
\begin{align}
\label{lab6}u_{0,n}(x)&=\sum_{j=1}^l V_n^j(0,x)+w_{0,n}^l(x),\\
\label{lab7}u_{1,n}(x)&=\sum_{j=1}^l \partial_t V_n^j(0,x)+w_{1,n}^l(x),
\end{align}
with
\begin{align}
\label{lab8}\limsup_{n\rightarrow\infty}\lVert S(t)(w_{0,n}^l,w_{1,n}^l)\rVert_{L_{t,x}^{\frac{2(d+1)}{d-2}}}\mathop{\longrightarrow}_{l\rightarrow\infty}0,
\end{align}
where $S(t)(u_0,u_1)$ denotes the solution of the wave equation with initial data $u(0)=u_0$ and $\partial_t u(0)=u_1$.
For every $l\geq 1$, we also have
\begin{align}
\nonumber\lefteqn{\lVert u_{0,n}\rVert_{\dot{H}^1}^2+\lVert u_{1,n}\rVert_{L^2}^2}&\\
\label{lab9}&\hspace{0.5in}=\sum_{j=1}^l \left(\lVert V^j_0\rVert_{\dot{H}^1}^2+\lVert V^j_1\rVert_{L^2}^2\right)+\lVert w_{0,n}^l\rVert_{\dot{H}^1}^2+\lVert w_{1,n}^l\rVert_{L^2}^2+o(1),\quad n\rightarrow\infty,
\end{align}
and, for $j\neq k$,
\begin{align}
\label{lab11}\lim_{n\rightarrow\infty} \lVert V_n^jV_n^k\rVert_{L_{t,x}^{\frac{d+1}{d-2}}}=0.
\end{align}
\end{theorem}

Applying the profile decomposition to an appropriate sequence, we obtain the existence of maximizers for ($\ref{lab2}$).
\begin{theorem}
\label{lab12}
Let $d\geq 3$, $(q,r)$ be a wave admissible pair with $q,r\in (2, \infty)$ and satisfying the $\dot{H}^1$-scaling condition.\footnote{We define the notion of the $\dot{H}^1$-scaling condition in Section $2$.}  Then there exists a maximizing pair $(\phi,\psi)\in \dot{H}^1\times L^2(\mathbb{R}^d)$ such that \begin{align*}\lVert S(t)(\phi,\psi)\rVert_{L_t^qL_x^r(\mathbb{R}\times\mathbb{R}^d)}=W_{q,r}\,\lVert (\phi, \psi)\rVert_{\dot{H}^1\times L^2(\mathbb{R}^d)}\end{align*} where \begin{align*} W_{q,r}&:=\sup \{\lVert S(t)(\phi,\psi)\rVert_{L_t^qL_x^r}:(\phi,\psi)\in \dot{H}^1\times L^2,\,\, \lVert (\phi, \psi)\rVert_{\dot{H}^1\times L^2(\mathbb{R}^d)}=1\}
\end{align*}
is the sharp constant.
\end{theorem}

Along the same lines, we consider maximizers of the inequalities \begin{align}\label{lab13} \lVert u(t,x)\rVert_{L_t^qL_x^r}\leq W'_{q,r}\lVert (u_0, u_1)\rVert_{\dot{H}^s\times\dot{H}^{s-1}(\mathbb{R}^d)},\end{align} where $(q,r)$ is a wave admissible pair with $q,r\in (2,\infty)$ and satisfying the $\dot{H}^s$-scaling condition, $s\geq 1$.  These inequalities are obtained by combining the Sobolev inequality with the Strichartz inequalities ($\ref{lab2}$). As in the case of $\dot{H}^1\times L^2$ initial data, the main tool is the corresponding linear profile decomposition.  More precisely, we prove the following version (for an identical statement see Lemma $4.9$ in Kenig-Merle \cite{KenigMerleSupercritical}):

\begin{theorem}
\label{lab14}Let $s\geq 1$ be given and let $(u_{0,n},u_{1,n})_{n\in\mathbb{N}}$ be a bounded sequence in $\dot{H}^s\times \dot{H}^{s-1}(\mathbb{R}^d)$ with $d\geq 3$.  Then there exists a subsequence of $(u_{0,n},u_{1,n})$ (still denoted $(u_{0,n},u_{1,n})$), a sequence $(V_0^j,V_1^j)_{j\in \mathbb{N}}\subset \dot{H}^s\times\dot{H}^{s-1}(\mathbb{R}^d)$, and a sequence of triples $(\epsilon_n^j,x_n^j,t_n^j)\in \mathbb{R^+}\times\mathbb{R}^d\times\mathbb{R}$ such that for every $j\neq j'$,
\begin{align*}
\frac{\epsilon_n^j}{\epsilon_n^{j'}}+\frac{\epsilon_n^{j'}}{\epsilon_n^j}+\frac{|t_n^j-t_n^{j'}|}{\epsilon_n^j}+\frac{|x_n^j-x_n^{j'}|}{\epsilon_n^j}\mathop{\longrightarrow}_{n\rightarrow\infty}
\infty,
\end{align*}
and for every $l\geq 1$, if $V^j=S(t)(V_0^j,V_1^j)$ and $V_n^j(t,x)=\frac{1}{(\epsilon_n^j)^{\frac{d-2}{2}-(s-1)}}V^j\left(\frac{t-t_n^j}{\epsilon_n^j},\frac{x-x_n^j}{\epsilon_n^j}\right)$,
\begin{align}
u_{0,n}(x)&=\sum_{j=1}^l V_n^j(0,x)+w_{0,n}^l(x),\label{lab15}\\
u_{1,n}(x)&=\sum_{j=1}^l \partial_t V_n^j(0,x)+w_{1,n}^l(x)\label{lab16},
\end{align}
with
\begin{align}
\limsup_{n\rightarrow\infty}\lVert S(t)(w_{0,n}^l,w_{1,n}^l)\rVert_{L_t^qL_x^r}\mathop{\longrightarrow}_{l\rightarrow\infty}0\label{lab17}
\end{align}
for every $(q,r)$ a wave admissible pair with $q,r\in (2,\infty)$ and satisfying the $\dot{H}^s$-scaling condition. For every $l\geq 1$, we also have
\begin{align}
\nonumber \lefteqn{\lVert u_{0,n}\rVert_{\dot{H}^s}^2+\lVert u_{1,n}\rVert_{\dot{H}^{s-1}}^2}&\\
&\hspace{0.5in}=\sum_{j=1}^l \left(\lVert V^j_0\rVert_{\dot{H}^s}^2+\lVert V^j_1\rVert_{\dot{H}^{s-1}}^2\right)+\lVert w_{0,n}^l\rVert_{\dot{H}^s}^2+\lVert w_{1,n}^l\rVert_{\dot{H}^{s-1}}^2+o(1),\quad n\rightarrow\infty,\label{lab18}.
\end{align}
\end{theorem}

As a consequence of Theorem $\ref{lab14}$, we immediately obtain the existence of maximizers for the inequalities ($\ref{lab13}$).

\begin{theorem}
\label{lab20}
Let $d\geq 3$, $s\geq 1$ be given and let $(q,r)$ be a wave admissible pair with $q,r\in (2,\infty)$ and satisfying the $\dot{H}^s$-scaling condition.  Then there exists a maximizing pair $(\phi,\psi)\in \dot{H}^s\times \dot{H}^{s-1}(\mathbb{R}^d)$ such that \begin{align*}\lVert S(t)(\phi,\psi)\rVert_{L_t^qL_x^r}=W'_{q,r}\lVert (\phi, \psi)\rVert_{\dot{H}^s\times\dot{H}^{s-1}(\mathbb{R}^d)}\end{align*} where \begin{align*} W'_{q,r}&:=\sup \{\lVert S(t)(\phi,\psi)\rVert_{L_t^qL_x^r}:(\phi,\psi)\in \dot{H}^s\times \dot{H}^{s-1}\quad\textrm{with}\quad \lVert (\phi, \psi) \rVert_{\dot{H}^s\times\dot{H}^{s-1}(\mathbb{R}^d)}=1\}
\end{align*}
is the sharp constant.
\end{theorem}

\subsection*{Organization of the paper} The paper is organized as follows. In Section 2, we introduce the notation that we shall use throughout the paper and give some preliminaries.  Section 3 is devoted to the detailed proof of the linear profile decomposition Theorem \ref{lab4} and the existence of maximizers for the Strichartz estimates ($\ref{lab2}$).  In Section 4, we give a proof of Theorem $\ref{lab14}$ by using the $\dot{H}^1\times L^{2}$ initial data case, Theorem $\ref{lab4}$.  We then obtain the existence of maximizers for the the inequalities ($\ref{lab13}$). In Appendix A, we provide a proof of a refined Sobolev inequality, while in Appendix B, we fill out the details of Theorem $\ref{lab12}$.

\section{NOTATION AND PRELIMINARIES}

We will often use the notation $x_n=o(1)$, $n\rightarrow\infty$ to denote the limit $\displaystyle\lim_{n\rightarrow\infty} x_n=0$.

We use $L_t^qL_x^r(\mathbb{R}\times\mathbb{R}^d)$ to denote the Banach space of functions $u:\mathbb{R}\times\mathbb{R}^d\rightarrow\mathbb{C}$ whose norms are \[ \lVert u\rVert_{L_t^qL_x^r(\mathbb{R}\times\mathbb{R}^d)}:=\lVert \lVert u\rVert_{L_x^r}\rVert_{L_t^q}=\left(\int_{\mathbb{R}} \left(\int_{\mathbb{R^d}} |u(t)|^rdx\right)^\frac{q}{r}dt\right)^\frac{1}{q}<\infty, \] with the usual convention when $q$ or $r$ is infinity.  In the case $q=r$ we abbreviate $L_t^qL_x^r$ by $L_{t,x}^q$.  The operator $\nabla$ will refer to the derivative in the space variable only.

We define the Fourier transform on $\mathbb{R}^d$ by \begin{align*} \hat{f}(\xi)&:=\int_{\mathbb{R}^d} e^{-ix\cdot\xi}f(x)dx. \end{align*}

For $s\in \mathbb{R}$, we define the fractional differentiation operator $|\nabla|^s$ by \begin{align*} \widehat{|\nabla|^sf}(\xi)&:= |\xi|^s\hat{f}(\xi).\end{align*}

These define the homogeneous Sobolev norms, \begin{align*} \lVert f\rVert_{\dot{H}^s(\mathbb{R}^d)}&:=\lVert |\nabla|^sf\rVert_{L^2(\mathbb{R}^d)}.\end{align*}

For $s\geq 1$, we will also use the product space $\dot{H}^s\times \dot{H}^{s-1}$ equipped with the norm,
\begin{align*}
\lVert (f,g)\rVert_{\dot{H}^s\times\dot{H}^{s-1}(\mathbb{R}^d)}=\left(\lVert f\rVert_{\dot{H}^s(\mathbb{R}^d)}^2+\lVert g\rVert_{\dot{H}^{s-1}(\mathbb{R}^d)}^2\right)^\frac{1}{2}.
\end{align*}

The solution operator for the initial value problem ($\ref{lab1}$) will be denoted by $S(t)$, which can also be written as \begin{align*} S(t)(u_0,u_1)=\cos (t\sqrt{-\Delta})u_0+\frac{\sin(t\sqrt{-\Delta})}{\sqrt{-\Delta}}u_1.\end{align*}

For $(u_0,u_1)\in \dot{H}^1\times L^2(\mathbb{R}^d)$, we define the energy by \begin{align*} E((u_0,u_1))&:=\frac{1}{2}\int |\nabla u_0|^2dx+\frac{1}{2}\int |u_1|^2dx,\end{align*}  which is conserved for solutions of ($\ref{lab1}$): for all $t\in\mathbb{R}$, \begin{align*}
E((u(t),\partial_tu(t)))=E((u_0,u_1)).\end{align*}

\begin{definition} (Admissible pairs) For $d\geq 3$, we say that $(q,r)$ is a wave admissible pair if $q,r\geq 2$, $(q,r,d)\neq (2,\infty,3)$, and \begin{align*}\frac{1}{q}+\frac{d-1}{2r}\leq\frac{d-1}{4}.\end{align*}

For $s\geq 0$, we also say that a wave admissible pair $(q,r)$ satisfies the $\dot{H}^s$-scaling condition if \begin{align*}\frac{1}{q}+\frac{d}{r}&=\frac{d}{2}-s.\end{align*}
\end{definition}
We now give the precise statement of the Strichartz inequalities ($\ref{lab2}$).
\begin{lemma} (Strichartz Estimates) \cite{GinibreVelo},\cite{KeelTao}
For $d\geq 2$, if $(q,r)$ is a wave admissible pair with $r<\infty$ and satisfying the $\dot{H}^s$-scaling condition, then there exists $C>0$ such that for every $(u_0,u_1)\in \dot{H}^s\times \dot{H}^{s-1}(\mathbb{R}^d)$, \begin{align}\lVert S(t)(u_0,u_1)\rVert_{L_t^qL_x^r(\mathbb{R}\times\mathbb{R}^d)}&\leq C\lVert (u_0, u_1)\rVert_{\dot{H}^s\times\dot{H}^{s-1}(\mathbb{R}^d)}.\label{lab23}\end{align}\label{lab21}
\end{lemma}
When $s=1$, we will often make use of these estimates in the following form: \begin{align}\label{lab22}\lVert S(t)(u_0,u_1)\rVert_{L_t^qL_x^r}\leq \tilde{C}\left[E((u_0,u_1))\right]^\frac{1}{2}.\end{align}
\section{MAXIMIZERS FOR THE  $\dot{H}^1\times{L}^{2}$-STRICHARTZ INEQUALITIES}

In the first part of this section, we extend the linear profile decomposition for the wave equation, previously obtained by Bahouri-Gerard \cite{BahouriGerard} for $d=3$, to dimensions $d\geq3$. We first prove the profile decomposition with the diagonal pair $q=r=\frac{2(d+1)}{d-2}$, based on the arguments given in \cite{BahouriGerard}, \cite{Keraani} and \cite{Gerard}.  We then obtain the decomposition for any suitable wave admissible pair $(q,r)$ by using an interpolation argument. In the second part, we give a proof of the existence of maximizers for the inequalities ($\ref{lab2}$), through the use of the linear profile decomposition stated above in Theorem $\ref{lab4}$, in the spirit of \cite{ShaoSchrodinger}.
\subsection{Linear Profile Decomposition}\ \\
\ \\
We begin by recalling some preliminaries that will be used throughout this subsection. For further reference, see for instance \cite{BahouriGerard} and \cite{Keraani}.

If $\sigma$ is a function on $\mathbb{R}^d$, we define $\sigma(D)$ by \begin{align*} \widehat{(\sigma(D)f)}(\xi)&=\sigma(\xi)\widehat{f}(\xi).\end{align*}

We use the space-time Fourier transform, \begin{align*}\widetilde{f}(\sigma,\eta)&=\int_\mathbb{R}\int_{\mathbb{R}^d} e^{-it\sigma-ix\cdot \eta}f(t,x)dxdt.\end{align*}

We also define the following norm on $L^2$:
\begin{align}\label{lab24}
I_k(f):=\left(\int_{2^k\leq |\xi|\leq 2^{k+1}} |\hat{f}(\xi)|^2d\xi\right)^{1/2},\quad \lVert f\rVert_B :=  \sup_{k\in\mathbb{Z}} I_k(f).
\end{align}

We now state a variant of the Sobolev inequality.  The proof is given in detail in Appendix A.

\begin{lemma}(A Refined Sobolev Inequality) \cite{Gerard}\label{lab25} For $d\geq 3$, there exists a constant $C\geq 0$ such that for every $u\in \dot{H}^1(\mathbb{R}^d)$, we have,
\begin{align}\lVert u \rVert_{L^p}\leq C \lVert \nabla u\rVert_{L^2}^{\frac {2}{p}} \lVert\nabla u\rVert_{ {B}}^{1-\frac {2}{p}}\label{lab26}\end{align}
where $\frac{1}{p}=\frac{1}{2}-\frac{1}{d}$ and $\lVert f\rVert_{B}$ is defined by $(\ref{lab24})$.
\end{lemma}

\begin{definition}
Let $(f_n)_{n\in\mathbb{N}}$ be a bounded sequence of functions in $L^2(\mathbb{R}^d)$, $d\geq 3$.
Given a sequence $(\epsilon_n)_{n\in\mathbb{N}}\subset\mathbb{R}^+$, we say that $(f_n)$ is $(\epsilon_n)$-oscillatory if \begin{align*} \limsup_{n\rightarrow\infty} \left(\int_{\epsilon_n|\xi|\leq \frac{1}{R}} |\hat{f_n}(\xi)|^2d\xi+\int_{\epsilon_n |\xi|\geq R} |\hat{f_n}(\xi)|^2d\xi\right)\mathop{\longrightarrow}_{R\rightarrow\infty} 0,\end{align*}
and $(f_n)$ is $(\epsilon_n)$-singular if, for all $b>a>0$, \begin{align} \int_{a\leq \epsilon_n |\xi|\leq b} |\hat{f_n}(\xi)|^2d\xi\mathop{\longrightarrow}_{n\rightarrow\infty} 0.\end{align}
\end{definition}

\begin{remark}If $(f_n)$ is $(\epsilon_n)$-oscillatory and $(g_n)$ is $(\epsilon_n)$-singular then by Plancherel's formula and Cauchy-Schwartz we have, \begin{align}
\label{lab27}\int_{\mathbb{R}^d} f_n(x)\overline{g_n(x)}dx\mathop{\longrightarrow}_{n\rightarrow\infty} 0.
\end{align}
This gives the identity,
\begin{align} \label{lab28}\lVert f_n+g_n\rVert_{L^2}^2=\lVert f_n\rVert_{L^2}^2+\lVert g_n\rVert_{L^2}^2+o(1),\quad n\rightarrow\infty.\end{align}
\end{remark}

The next proposition provides a decomposition of bounded sequences in $L^2(\mathbb{R}^d)$.  For a detailed proof we refer the reader to Theorem $2.9$ in \cite{Gerard}.  We note that the proof given there uses a slightly different but equivalent norm on the space $B$.
\begin{proposition} \label{prop1} \cite{Gerard} Let $(f_n)_{n\in\mathbb{N}}$ be a bounded sequence in $L^2(\mathbb{R}^d)$ with $d\geq 3$.  Then there exists a subsequence of $(f_n)$ (still  denoted $(f_n)$), a sequence $(\epsilon_n^{j})\subset \mathbb{R}^+$ such that, for every $j\neq j'$, \begin{align} \frac{\epsilon_n^j}{\epsilon_n^{j'}}+\frac{\epsilon_n^{j'}}{\epsilon_n^j}\mathop{\longrightarrow}_{n\rightarrow\infty}\infty,\label{lab29}\end{align} and a bounded sequence $(g_n^j)\subset L^2(\mathbb{R}^d)$, such that for every $l\geq 1, x\in \mathbb{R}^d$, \begin{align*} f_n(x)&=\sum_{j=1}^l g_n^{j}(x)+r_n^{l}(x),\end{align*} where $(g_n^j)$ is $(\epsilon_n^j)$-oscillatory, $(r_n^{l})$ is $(\epsilon_n^{j})$-singular, $1\leq j\leq l$, and \begin{align*}\limsup_{n\rightarrow\infty} \lVert r_n^{l}\rVert_B\mathop{\longrightarrow}_{l\rightarrow\infty}0.\end{align*}
\end{proposition}
\begin{remark}  This proposition is similar to Proposition $3.4$ in \cite{BahouriGerard}.  Let us briefly point out the differences between the two statements.  The decomposition in \cite{BahouriGerard} is stated in the form
\begin{align*}
f_n=f+\sum_{j=1}^l g_n^j+r_n^l,
\end{align*}
while at the same time the result requires that the constructed sequence $(\epsilon_n^j)\subset\mathbb{R}^+$ satisfies $\displaystyle \epsilon_n^j\mathop{\longrightarrow}_{n\rightarrow\infty}0$ for each $j\geq 1$.  We emphasize that the statement we use does not require this condition on $(\epsilon_n^j)$.  This is the distinction which allows one to obtain the different form of the decomposition.  
\end{remark}

We now state an inequality from \cite{Gerard}.
\begin{lemma}\label{lab30}
For all $p\in [2,\infty)$ we have,
\begin{align*}
\left| \left|\sum_{j=1}^l a_j\right|^p-\sum_{j=1}^l |a_j|^p\right|&\leq C_l\sum_{j\neq k} |a_j||a_k|^{p-1}.\end{align*}
\end{lemma}

In the proof of the profile decomposition we will often use the following fact, which we prove in a similar spirit to Lemma $2.7$ in \cite{Keraani}.
\begin{lemma} \label{lab31}Let $(\epsilon_n^{j},x_n^{j},t_n^j)\subset\mathbb{R}^+\times\mathbb{R}^d\times\mathbb{R}$, $d\geq 3$, be a sequence of triples satisfying ($\ref{lab5}$) and $(V^{j})$ be a sequence of functions in $L^{\frac{2(d+1)}{d-2}}_{t,x}(\mathbb{R}^{d+1})$.  Define \begin{align*}
V_n^{j}(t,x):=\frac{1}{(\epsilon_n^{j})^\frac{d-2}{2}}V^{j}\left(\frac{t-t_n^{j}}{\epsilon_n^{j}},\frac{x-x_n^{j}}{\epsilon_n^{j}}\right).
\end{align*}

Then for every $l\geq 1$ we have
\begin{align*}\left\lVert \sum_{j=1}^l V_n^{j}\right\rVert_{L_{t,x}^\frac{2(d+1)}{d-2}}^\frac{2(d+1)}{d-2}\mathop{\longrightarrow}_{n\rightarrow\infty} \sum_{j=1}^l \left\lVert V^{j}\right\rVert_{L_{t,x}^\frac{2(d+1)}{d-2}}^\frac{2(d+1)}{d-2}.\end{align*}
\end{lemma}

\begin{proof}
Note that, for $n\geq 1$,
\begin{align*}
\lVert V_n^{j}\rVert_{L_{t,x}^\frac{2(d+1)}{d-2}}^\frac{2(d+1)}{d-2}=\int\int\left|\frac{1}{(\epsilon_n^{j})^\frac{d-2}{2}}V^{j}\left(\frac{t-t_n^{j}}{\epsilon_n^{j}},\frac{x-x_n^{j}}{\epsilon_n^{j}}\right)\right|^\frac{2(d+1)}{d-2}dxdt=\lVert V^{j}\rVert_{L_{t,x}^\frac{2(d+1)}{d-2}}^\frac{2(d+1)}{d-2}.
\end{align*}

Then, using Lemma $\ref{lab30}$,
\begin{align*}
\left|\left\lVert \sum_{j=1}^l V_n^{j}\right\rVert_{L_{t,x}^\frac{2(d+1)}{d-2}}^\frac{2(d+1)}{d-2}-\sum_{j=1}^l \left\lVert V^{j}\right\rVert_{L_{t,x}^\frac{2(d+1)}{d-2}}^\frac{2(d+1)}{d-2}\right|&\leq \int\int \left|\left|\sum_{j=1}^l V_n^{j}\right|^\frac{2(d+1)}{d-2}-\sum_{j=1}^l |V_n^{j}|^\frac{2(d+1)}{d-2}\right|dxdt\\
&\leq C_l\int\int\sum_{j\neq k} |V_n^{j}||V_n^{k}|^{\frac{d+4}{d-2}}dxdt.
\end{align*}

Let $j\neq k$ be given.
Note that ($\ref{lab5}$) shows that either
\begin{align}
\label{lab32}
\frac{\epsilon_n^{j}}{\epsilon_n^{k}}+\frac{\epsilon_n^{k}}{\epsilon_n^{j}}\mathop{\longrightarrow}_{n\rightarrow\infty}\infty,
\end{align}
or
\begin{align}
\label{lab33}
\textrm{for all}\,\,n,\quad \epsilon_n^j=\epsilon_n^k\quad\textrm{and}\quad
\left|\frac{t_n^{j}-t_n^{k}}{\epsilon_n^{j}}\right|+\left|\frac{x_n^{j}-x_n^{k}}{\epsilon_n^{j}}\right|\mathop{\longrightarrow}_{n\rightarrow\infty}\infty.
\end{align}

Without loss of generality, we assume that $V^{j}$ and $V^{k}$ are compactly supported.

\underline{\bf Case 1}: Suppose ($\ref{lab32}$) holds.  Then either $\displaystyle \frac{\epsilon_n^{k}}{\epsilon_n^{j}}\mathop{\longrightarrow}_{n\rightarrow\infty}\infty$ or $\displaystyle \frac{\epsilon_n^{j}}{\epsilon_n^{k}}\mathop{\longrightarrow}_{n\rightarrow\infty}\infty$.  Suppose $\displaystyle\frac{\epsilon_n^{k}}{\epsilon_n^{j}}\mathop{\longrightarrow}_{n\rightarrow\infty}\infty$ (the proof for the other case is identical).

Then, using the change of variables $x\mapsto \epsilon_n^{j}x+x_n^{j}$, $t\mapsto \epsilon_n^{j}t+t_n^{j}$, we get
\begin{align}
\label{lab34}\int\int \lefteqn{|V_n^{j}||V_n^{k}|^{\frac{d+4}{d-2}}dxdt}\\
\nonumber&= \int\frac{(\epsilon_n^{j})^d\epsilon_n^{j}}{(\epsilon_n^{j})^\frac{d-2}{2}(\epsilon_n^{k})^\frac{d+4}{2}}|V^{j}(t,x)|\left|V^{k}\left(\frac{(\epsilon_n^{j})t}{\epsilon_n^{k}}+\frac{t_n^{j}-t_n^{k}}{\epsilon_n^{k}},\frac{\epsilon_n^{j}x}{\epsilon_n^{k}}+\frac{x_n^{j}-x_n^{k}}{\epsilon_n^{k}}\right)\right|^{\frac{d+4}{d-2}}dxdt,\\
\nonumber&\leq C\left(\frac{\epsilon_n^{j}}{\epsilon_n^{k}}\right)^{\frac{d+4}{2}},
\end{align}
where we use the fact that $V^j$ and $V^k$ are continuous and compactly supported.
Note that $\displaystyle \frac{\epsilon_n^{k}}{\epsilon_n^{j}}\mathop{\longrightarrow}_{n\rightarrow\infty}\infty$ implies $\displaystyle \left(\frac{\epsilon_n^{j}}{\epsilon_n^{k}}\right)^\frac{d+4}{2}\mathop{\longrightarrow}_{n\rightarrow\infty} 0$.  Thus, in this case, $\displaystyle (\ref{lab34})\mathop{\longrightarrow}_{n\rightarrow\infty} 0$.

\underline{\bf Case 2}: Suppose ($\ref{lab33}$) holds.  Then
\begin{align}
\label{lab36} \lefteqn{\int\int |V_n^{j}||V_n^{k}|^{\frac{d+4}{d-2}}dxdt}\\
\nonumber &= \int\int \left(\frac{\epsilon_n^{j}}{\epsilon_n^{k}}\right)^\frac{d+4}{2}|V^{j}(t,x)|\left|V^{k}\left(\frac{\epsilon_n^{j}t}{\epsilon_n^{k}}+\frac{t_n^{j}-t_n^{k}}{\epsilon_n^{k}},\frac{\epsilon_n^{j}x}{\epsilon_n^{k}}+\frac{x_n^{j}-x_n^{k}}{\epsilon_n^{k}}\right)\right|^{\frac{d+4}{d-2}}dxdt\\
\nonumber &= \int\int |V^{j}(t,x)|\left|V^{k}\left(t+\frac{t_n^{j}-t_n^{k}}{\epsilon_n^{k}},x+\frac{x_n^{j}-x_n^{k}}{\epsilon_n^{k}}\right)\right|^{\frac{d+4}{d-2}}dxdt,
\end{align}
using $\epsilon_n^{j}=\epsilon_n^{k}$ for all $n$.  Since $V^{j}$ and $V^{k}$ are continuous and have compact support,
\begin{align*}
\lefteqn{|V^{j}(t,x)|\left|V^{k}\left(t+\frac{t_n^{j}-t_n^{k}}{\epsilon_n^{k}},x+\frac{x_n^{j}-x_n^{k}}{\epsilon_n^{k}}\right)\right|^{\frac{d+4}{d-2}}\leq}\\
&\hspace{1.5in}\chi_{\supp V^{j}}\,\,\lVert V^{j}\rVert_{L_{t,x}^\infty}\lVert V^{k}\rVert_{L_{t,x}^\infty}^{\frac{d+4}{d-2}}\in L^1(\mathbb{R}^{d+1}).
\end{align*}

Note that, by assumption, $\left|\frac{t_n^{j}-t_n^{k}}{\epsilon_n^{k}}\right|+\left|\frac{x_n^{j}-x_n^{k}}{\epsilon_n^{k}}\right|\rightarrow\infty$.  Then $V^{k}$ having compact support implies \begin{align*}
V^{k}\left(t+\frac{t_n^{j}-t_n^{k}}{\epsilon_n^{k}},x+\frac{x_n^{j}-x_n^{k}}{\epsilon_n^{k}}\right)\rightarrow 0
\end{align*}
 for all $(t,x)\in\mathbb{R}^{d+1}$.  Thus, by Lebesgue's Dominated Convergence Theorem, $\displaystyle (\ref{lab36})\mathop{\longrightarrow}_{n\rightarrow\infty} 0$.

Thus for every $j\neq k$, in each case, $\int\int |V_n^j||V_n^k|^\frac{d+4}{d-2}dxdt\rightarrow 0$.  This gives the result, \begin{align*}\int\int\sum_{j\neq k} |V_n^j||V_n^k|^\frac{d+4}{d-2}dxdt\mathop{\longrightarrow}_{n\rightarrow\infty} 0,\end{align*} which leads to desired claim.
\end{proof}
We are now ready to prove Theorem $1$.
\subsubsection{Proof of Theorem $1$}

Applying Proposition $\ref{prop1}$ to the sequences $(\partial_k u_{0,n})$, $k=1,\cdots,d$, and $(u_{1,n})$, we obtain a subsequence $(u_{0,n},u_{1,n})$, a sequence $(\epsilon_n^{j})\in\mathbb{R}^+$ satisfying ($\ref{lab29}$) and for every $j$, a sequence $(p_{0,n}^j,p_{1,n}^j)$ bounded in $\dot{H}^1\times L^2(\mathbb{R}^d)$ such that for every $l\geq 1$,
\begin{align} (u_{0,n},u_{1,n})=\sum_{j=1}^l (p_{0,n}^j,p_{1,n}^j)+(r_n^{l},s_n^{l}),\label{lab37} \end{align}
where $(\nabla p_{0,n}^j,p_{1,n}^j)$ is $(\epsilon_n^{j})$-oscillatory and $(\nabla r_n^{l},s_n^{l})$ is $(\epsilon_n^j)$-singular, for $1\leq j\leq l$ and
\begin{align} \label{lab38}\limsup_{n\rightarrow\infty} (\lVert \nabla r_n^{l}\rVert_{B}+\lVert s_n^{l}\rVert_B)\mathop{\longrightarrow}_{l\rightarrow\infty} 0.\end{align}

Moreover the identity ($\ref{lab28}$) implies, for every $l\geq 1$,
\begin{align}
\label{lab40}\lVert u_{0,n}\rVert_{\dot{H}^1}^2&= \sum_{j=1}^l \lVert p_{0,n}^j\rVert_{\dot{H}^1}^2+\lVert r_n^l\rVert_{\dot{H}^1}^2+o(1),\quad n\rightarrow\infty,\\
\label{lab41}\lVert u_{1,n}\rVert_{L^2}^2&= \sum_{j=1}^l \lVert p_{1,n}^j\rVert_{L^2}^2+\lVert s_n^l\rVert_{L^2}^2+o(1),\quad n\rightarrow\infty.
\end{align}

From ($\ref{lab37}$), we get the corresponding decomposition,
\begin{align}\label{lab42}S(t)(u_{0,n},u_{1,n})=\sum_{j=1}^l S(t)(p_{0,n}^j,p_{1,n}^j)+S(t)(r_n^l,s_n^l).\end{align}
\setcounter{secnumdepth}{6}
\paragraph{\bf Estimation of the remainder term $S(t)(r_n^l,s_n^l)$.}\ \\
\ \\
Note that, if $q$ is a finite energy solution to $(\partial_{tt}-\Delta) q=0$ then $\sigma_k(D)q$ is also a solution to the same equation, where $\sigma_k(\xi)=\chi_{2^k\leq|\xi|\leq 2^{(k+1)}}(\xi)$. Then the conservation law for all $\sigma_k(D)S(t)(r_n^l,s_n^l)$, $k\in\mathbb{Z}$, implies for some $C>0$,
\begin{align}
\label{lab43}
\left(\lVert\nabla S(t)(r_n^l,s_n^l)\rVert^2_{L_t^{\infty}B}+\lVert\partial_t S(t)(r_n^l,s_n^l)\rVert^2_{L_t^\infty B}\right)^\frac{1}{2}&\leq C(\lVert \nabla r_n^l \rVert_B+\lVert s_n^l\rVert_B)
\end{align}
Putting together ($\ref{lab38}$) and ($\ref{lab43}$), it follows that
\begin{align}
\limsup_{n\rightarrow\infty} \left(\lVert\nabla S(t)(r_n^l,s_n^l)\rVert^2_{L_t^{\infty} B}+\lVert \partial_t S(t)(r_n^l,s_n^l)\rVert^2_{L_t^\infty B}\right)^\frac{1}{2}\mathop{\longrightarrow}_{l\rightarrow\infty}0.
\label{lab44}
\end{align}

Now, applying Lemma $\ref{lab25}$ to $S(t)(r_n^l,s_n^l)$ and using ($\ref{lab40}$)-($\ref{lab41}$), we obtain the estimate,
\begin{align*}
\limsup_{n\rightarrow\infty}\lVert S(t)(r_n^l,s_n^l)\rVert_{L_t^\infty L_x^{\frac{2d}{d-2}}}&\leq C \, \limsup_{n\rightarrow\infty} \left[E((r_n^l,s_n^l))\right]^\frac{d-2}{2d} \limsup_{n\rightarrow\infty} \lVert\nabla S(t)(r_n^l,s_n^l)\rVert_{L_t^\infty B}^{\frac{2}{d}}\\
&\leq C \, \limsup_{n\rightarrow\infty} \left[E((u_{0,n},u_{1,n}))\right]^\frac{d-2}{2d} \limsup_{n\rightarrow\infty} \lVert\nabla S(t)(r_n^l,s_n^l)\rVert_{L_t^\infty B}^{\frac{2}{d}}.
\end{align*}
Then, the limit ($\ref{lab44}$) allows us to conclude that
\begin{align}
\label{lab44.5}\limsup_{n\rightarrow\infty} \lVert S(t)(r_n^l,s_n^l)\rVert_{L_t^{\infty} L_x^{\frac{2d}{d-2}}}\mathop{\longrightarrow}_{l\rightarrow\infty}0.
\end{align}

Moreover, by interpolation, the Strichartz inequality ($\ref{lab23}$) and ($\ref{lab40}$)-($\ref{lab41}$), we observe that
\begin{align}
\nonumber \lefteqn{\limsup_{n\rightarrow\infty} \lVert S(t)(r_n^l,s_n^l)\rVert_{L_{t,x}^{\frac{2(d+1)}{d-2}}}}&\\
\nonumber &\hspace{0.5in}\leq \limsup_{n\rightarrow\infty} \lVert S(t)(r_n^l,s_n^l)\rVert_{L_t^\frac{2d-1}{d-2}L_x^\frac{2d(2d-1)}{(2d-3)(d-2)}}^\frac{2d-1}{2(d+1)}\limsup_{n\rightarrow\infty} \lVert S(t)(r_n^l,s_n^l)\rVert_{L_t^\infty L_x^\frac{2d}{d-2}}^\frac{3}{2(d+1)}\\
\nonumber &\hspace{0.5in}\leq \limsup_{n\rightarrow\infty} E((r_n^l,s_n^l))^\frac{2d-1}{4(d+1)}\limsup_{n\rightarrow\infty} \lVert S(t)(r_n^l,s_n^l)\rVert_{L_t^\infty L_x^\frac{2d}{d-2}}^\frac{3}{2(d+1)}\\
\nonumber &\hspace{0.5in}\leq \limsup_{n\rightarrow\infty} E((u_{0,n},u_{1,n}))^\frac{2d-1}{4(d+1)}\limsup_{n\rightarrow\infty} \lVert S(t)(r_n^l,s_n^l)\rVert_{L_t^\infty L_x^\frac{2d}{d-2}}^\frac{3}{2(d+1)}.
\end{align}

Combining this with ($\ref{lab44.5}$),
\begin{align}
\label{lab45} \limsup_{n\rightarrow\infty} \lVert S(t)(r_n^l,s_n^l)\rVert_{L_{t,x}^\frac{2(d+1)}{d-2}}\mathop{\longrightarrow}_{l\rightarrow\infty} 0.
\end{align}

\paragraph{\bf Decomposition of the terms $S(t)(p_{0,n}^{j},p_{1,n}^j)$.}\ \\
\ \\
In a similar way as in \cite{BahouriGerard}, we decompose each term $S(t)(p_{0,n}^j,p_{1,n}^j)$ through the following lemma.
\begin{lemma}
\label{lab46}
\label{lab47}Let $(p_{0,n},p_{1,n})_{n\in\mathbb{N}}$ be a bounded sequence in $\dot{H}^1\times L^2(\mathbb{R}^d)$ with $d\geq 3$, such that $(\nabla p_{0,n},p_{1,n})$ is $(\epsilon_n)$-oscillatory for some sequence $(\epsilon_n)$ in $\mathbb{R}^+$.  Then there exists a subsequence of $(p_{0,n},p_{1,n})$ (still denoted $(p_{0,n},p_{1,n})$), a sequence $(V_0^\alpha,V_1^\alpha)_{\alpha\in\mathbb{N}}$, and pairs $(t_n^\alpha,x_n^\alpha)\in \mathbb{R}\times\mathbb{R}^d$ such that for every $\alpha\neq\alpha'$,
\begin{align}
\label{lab47.5}
\frac{|t_n^\alpha-t_n^\beta|}{\epsilon_n}+\frac{|x_n^\alpha-x_n^\beta|}{\epsilon_n}\mathop{\longrightarrow}_{n\rightarrow\infty} \infty,
\end{align}
and for every $A\geq 1$, if $V^\alpha=S(t)(V_0^\alpha,V_1^\alpha)$ and $V_n^\alpha(t,x)=\frac{1}{(\epsilon_n)^\frac{d-2}{2}}V^\alpha\left(\frac{t-t_n^\alpha}{\epsilon_n},\frac{x-x_n^\alpha}{\epsilon_n}\right)$, then for every $x\in\mathbb{R}^d$,
\begin{align}
\label{lab48}p_{0,n}(x)&=\sum_{\alpha=1}^A V_n^\alpha(0,x)+\rho_{0,n}^A(x),\\
\label{lab49}p_{1,n}(x)&=\sum_{\alpha=1}^A \partial_t V_n^\alpha(0,x)+\rho_{1,n}^A(x),
\end{align}
with \begin{align}
\label{lab50}\limsup_{n\rightarrow\infty} \lVert S(t)(\rho_{0,n}^A,\rho_{1,n}^A)\rVert_{L_{t,x}^\frac{2(d+1)}{(d-1)}}\mathop{\longrightarrow}_{A\rightarrow\infty}0.
\end{align}
We also have,
\begin{align}
\nonumber\lefteqn{\lVert p_{0,n}\rVert_{\dot{H}^1}^2+\lVert p_{1,n}\rVert_{L^2}^2}&\\
\label{lab51}&\hspace{0.5in}=\sum_{\alpha=1}^A \left(\lVert V_0^\alpha\rVert_{\dot{H}^1}^2+\lVert V_1^\alpha\rVert_{L^2}^2\right)+\lVert \rho_{0,n}^A\rVert_{\dot{H}^1}^2+\lVert \rho_{1,n}^A\rVert_{L^2}^2+o(1),\quad n\rightarrow\infty.
\end{align}
\end{lemma}

\begin{proof}[Proof of Lemma $\ref{lab46}$]
For each $(P_{0,n},P_{1,n})_{n\in\mathbb{N}}$ a bounded sequence in $\dot{H}^1\times L^2(\mathbb{R}^d)$, define \begin{align*}
\mathcal{V}((P_{0,n},P_{1,n}))=\left\{\,\,\begin{array}{l}(V_0,V_1)\\\in\dot{H}^1\times L^2\end{array}\,\,\,\middle|\begin{array}{c}\textrm{ there exists a sequence }\,(s_n,y_n)\subset \mathbb{R}\times\mathbb{R}^d\\
\textrm{such that, up to a subsequence, }\\\displaystyle S(s_n)(P_{0,n},P_{1,n})(y+y_n)\mathop{\rightharpoonup}_{n\rightarrow\infty} V_0\,\textrm{weakly in }\,\dot{H}^1\quad\,{and}\\\displaystyle \partial_t S(s_n)(P_{0,n},P_{1,n})(y+y_n)\mathop{\rightharpoonup}_{n\rightarrow\infty} V_1\,\textrm{weakly in }\,L^2\end{array}\right\},
\end{align*}
and
\begin{align*}\eta((P_{0,n},P_{1,n}))=\sup\{\left[E((V_0,V_1))\right]^\frac{1}{2}: (V_0,V_1)\in \mathcal{V}((P_{0,n},P_{1,n}))\}.\end{align*}

Note that if we take $s_n=0$ and $y_n=0$, the boundedness of $(P_{0,n},P_{1,n})$ implies that $S(0)(P_{0,n},P_{1,n})(y)$ and $\partial_t S(0)(P_{0,n},P_{1,n})(y)$ have weakly convergent subsequences in $\dot{H}^1(\mathbb{R}^d)$ and $L^2(\mathbb{R}^d)$ respectively, and thus $\mathcal{V}((P_{0,n},P_{1,n}))$ is nonempty.

Note that for each such $(P_{0,n},P_{1,n})$, the weak lower semicontinuity of the norm implies
\begin{align*}
\eta((P_{0,n},P_{1,n}))\leq \limsup_{n\rightarrow\infty}\left[E((P_{0,n},P_{1,n}))\right]^\frac{1}{2}.\end{align*}

For $n\in\mathbb{N}$, put $P_{0,n}(x)=\epsilon_n^\frac{d-2}{2}p_{0,n}(\epsilon_nx)$, $P_{1,n}(x)=\epsilon_n^{\frac{d-2}{2}+1}p_{1,n}(\epsilon_nx)$ for $x\in\mathbb{R}^d$.

The proof proceeds in two steps.  In Step $1$, we obtain the desired decomposition with a weaker version of the smallness condition ($\ref{lab50}$).  Then Step $2$ completes the argument by providing the desired condition ($\ref{lab50}$).

{\bf\underline{Step 1}}: (Decomposition)  We claim that there exists a subsequence of $(P_{0,n},P_{1,n})$ (still denoted $(P_{0,n},P_{1,n})$), a sequence $(V_0^\alpha,V_1^\alpha)_{\alpha\in\mathbb{N}}$, and pairs $(s^\alpha,y^\alpha)\subset\mathbb{R}\times\mathbb{R}^d$ such that for every $\alpha\neq\alpha'$,
\begin{align}
\label{lab53}|s_n^\alpha-s_n^{\alpha'}|+|y_n^\alpha-y_n^{\alpha'}|\mathop{\longrightarrow}_{n\rightarrow\infty} \infty,
\end{align}
and for every $A\geq 1$, if $V^\alpha=S(t)(V_0^\alpha,V_1^\alpha)$ and $V_n^\alpha=V^\alpha(s-s_n^\alpha,y-y_n^\alpha)$,
\begin{align}\label{lab54}P_{0,n} (y)=\sum_{\alpha=1}^A V_n^{\alpha}(0, y)+P_{0,n}^{A}(y),\\
\label{lab55}P_{1,n}(y)=\sum_{\alpha=1}^A \partial_t V_n^\alpha(0,y)+P_{1,n}^A(y)
\end{align}
with
\begin{align*}
\lim_{A\rightarrow\infty} \eta((P_{0,n}^A,P_{1,n}^A))=0.
\end{align*}
and
\begin{align}
\nonumber\lefteqn{\lVert P_{0,n}\rVert_{\dot{H}^1}^2+\lVert P_{1,n}\rVert_{L^2}^2}&\\
\label{lab56}&\hspace{0.5in}=\sum_{\alpha=1}^A \left(\lVert V_0^\alpha\rVert_{\dot{H}^1}^2+\lVert V_1^\alpha\rVert_{L^2}^2\right)+\lVert P_{0,n}^A\rVert_{\dot{H}^1}^2+\lVert P_{1,n}^A\rVert_{L^2}^2+o(1)\quad n\rightarrow\infty.
\end{align}

To see this claim, note that if $\eta((P_{0,n},P_{1,n}))=0$, we can take $V_0^{\alpha}=0$, $V_1^\alpha=0$ for all $\alpha$.  On the other hand, if $\eta((P_{0,n},P_{1,n}))>0$, then there exists $(V_0^{1},V_1^1)\in \mathcal{V}((P_{0,n},P_{1,n}))$
such that
\[\left[E((V_0^1,V_1^1))\right]^\frac{1}{2}\geq\frac{1}{2}\eta((P_{0,n},P_{1,n}))>0.\]

Then, by the definition of the set, we can choose a sequence $(s_n^1, y_{n}^1) \subset\mathbb{R}\times\mathbb{R}^d$ such that, up to a subsequence,
\begin{align}\label{lab58}S(s_n^1)(P_{0,n},P_{1,n})(y+y_n^1)\mathop{\rightharpoonup}_{n\rightarrow\infty} V_0^1\quad\textrm{weakly in}\,\dot{H}^1,\\
\partial_t S(s_n^1)(P_{0,n},P_{1,n})(y+y_n^1)\mathop{\rightharpoonup}_{n\rightarrow\infty} V_1^1\quad\textrm{weakly in}\,L^2,\end{align}
and we set
\begin{align*}P_{0,n}^1(y):=P_{0,n}(y)-S(-s_n^1)(V_0^1,V_1^1)(y-y_n^1),\\ P_{1,n}^1(y):=P_{1,n}(y)-\partial_t S(-s_n^1)(V_0^1,V_1^1)(y-y_n^1).\end{align*}

Then, using the conservation of energy and noting that
\begin{align*}
S(s_n^1)(P_{0,n}^1,P_{1,n}^1)(y+y_n^1)\mathop{\rightharpoonup}_{n\rightarrow\infty} 0\quad\textrm{and}\quad\partial_t S(s_n^1)(P_{0,n},P_{1,n})(y+y_n^1)\mathop{\rightharpoonup}_{n\rightarrow\infty} 0,
\end{align*}
 we obtain the identity
\begin{align}
\nonumber\lefteqn{\lVert P_{0,n}\rVert_{\dot{H}^1}^2+\lVert P_{1,n}\rVert_{L^2}^2}&\\
\nonumber&\hspace{0.2in}=\big\lVert S(-s_n^1)(V_0^1,V_1^1)(\cdot-y_n^1)\big\rVert_{\dot{H}^1}^2+\big\lVert \partial_t S(-s_n^1)(V_0^1,V_1^1)(\cdot-y_n^1)\big\rVert_{L^2}^2+\big\lVert P_{0,n}^1\big\rVert_{\dot{H}^1}^2+\big\lVert P_{1,n}^1\big\rVert_{L^2}^2\\
\nonumber&\hspace{0.8in}+2\big\langle P_{0,n}^1,S(-s_n^1)(V_0^1,V_1^1)(\cdot-y_n^1)\big\rangle_{\dot{H}^1}+2\big\langle P_{1,n}^1,\partial_t S(-s_n^1)(V_0^1,V_1^1)(\cdot-y_n^1)\big\rangle_{L^2}\\
\nonumber&\hspace{0.2in}=\big\lVert V_0^1\big\rVert_{\dot{H}^1}^2+\big\lVert V_1^1\big\rVert_{L^2}^2+\big\lVert P_{0,n}^1\big\rVert_{\dot{H}^1}^2+\big\lVert P_{1,n}^1\big\rVert_{L^2}^2\\
\nonumber&\hspace{0.8in}+2\big\langle S(s_n^1)(P_{0,n}^1,P_{1,n}^1)(\cdot+y_n^1),V_0^1\big\rangle_{\dot{H}^1}+2\big\langle \partial_t S(s_n^1)(P_{0,n}^1,P_{1,n}^1)(\cdot+y_n^1),V_1^1\big\rangle_{L^2}\\
&\hspace{0.2in}=\lVert V_0^1\rVert_{\dot{H}^1}^2+\lVert V_1^1\rVert_{L^2}^2+\lVert P_{0,n}^1\rVert_{\dot{H}^1}^2+\lVert P_{1,n}^1\rVert_{L^2}^2+o(1),\quad n\rightarrow\infty.
\label{lab58.5}
\end{align}

Now we can repeat the above process while replacing $(P_{0,n},P_{1,n})$ with the pair $(P_{0,n}^1,P_{1,n}^1)$.  If $\eta((P_{0,n}^1,P_{1,n}^1))>0$ then, in this case, we get $V_0^{2},V_1^2$, $(s_n^{2},y_n^{2})$ and $(P_{0,n}^2,P_{1,n}^2)$. Moreover, \[|s_n^{1}-s_n^{2}|+|y_n^{1}-y_n^{2}|\mathop{\longrightarrow}_{n\rightarrow\infty}\infty.\]
 If not, then we can find a subsequence (still indexed by $n$) such that
\begin{align*}
s_n^{1}-s_n^{2}=s^*+a_n,\quad a_n\rightarrow 0\quad\textrm{and}\quad y_n^1-y_n^2=y^*+b_n,\quad b_n\rightarrow 0.
\end{align*}
so that for every $(h_1,h_2)\in \dot{H}^1\times L^2(\mathbb{R}^d)$ a change of variables followed by using the strong convergence of $S(s^*+a_n)(h_1,h_2)(\cdot+y^*+b_n)$ and $\partial_t S(s^*+a_n)(h_1,h_2)(\cdot+y^*+b_n)$ and the weak convergence of $S(s_n^1)(P_{0,n}^1,P_{1,n}^1)(\cdot+y_n^1)$ and $\partial_t S(s_n^1)(P_{0,n}^1,P_{1,n}^1)(\cdot+y_n^1)$ imply that
\begin{align*}
\lefteqn{\langle S(s_n^2)(P_{0,n}^1,P_{1,n}^1)(\cdot+y_n^2),h_1\rangle_{\dot{H}^1}+\langle \partial_t S(s_n^2)(P_{0,n}^1,P_{1,n}^1)(\cdot+y_n^2),h_2\rangle_{L^2}}&\\
&\hspace{0.2in}=\langle S(s_n^1)(P_{0,n}^1,P_{1,n}^1)(\cdot+y_n^1),S(s^*+a_n)(h_1,h_2)(\cdot+y^*+b_n)\rangle_{\dot{H}^1}\\
&\hspace{0.4in}+\langle \partial_t S(s_n^1)(P_{0,n}^1,P_{1,n}^1)(\cdot+y_n^1),\partial_t S(s^*+a_n)(h_1,h_2)(\cdot+y^*+b_n)\rangle_{L^2}\\
&\hspace{0.2in}\rightarrow \langle 0,S(s^*)(h_1,h_2)(\cdot+y^*)\rangle_{\dot{H}^1}+\langle 0,S(s^*)(h_1,h_2)(\cdot+y^*)\rangle_{L^2}=0.
\end{align*}
Then, recalling that $S(s_n^2)(P_{0,n}^1,P_{1,n}^1)(\cdot+y_n^2)\rightharpoonup V_0^2$, $\partial_t S(s_n^2)(P_{0,n}^1,P_{1,n}^1)(\cdot+y_n^2)\rightharpoonup V_1^2$, the uniqueness of weak limits would imply that $V_0^2=0$ and $V_1^2=0$, and therefore $\eta ((P_{0,n}^1,P_{1,n}^1))=0$, which gives a contradiction.
By iterating this process and using a diagonal extraction, we obtain the sequence $(V_0^\alpha,V_1^{\alpha})$ and pairs $(s_n^{\alpha},y_n^{\alpha})$ satisfying the claims ($\ref{lab53}$)-($\ref{lab56}$). Moreover, ($\ref{lab56}$)  implies the convergence of the series $\displaystyle\sum_{\alpha} \left[E((V_0^\alpha,V_1^\alpha))\right]^\frac{1}{2}$, and hence
\[\lim_{\alpha\rightarrow\infty} \left[E((V_0^\alpha,V_1^\alpha))\right]^\frac{1}{2}=0.\]
In addition, the construction gives the inequality,
\[\left[E((V_0^\alpha,V_1^\alpha))\right]^\frac{1}{2}\geq\frac{1}{2}\eta((P_{0,n}^{\alpha-1},P_{1,n}^{\alpha-1})),\] which implies that \begin{align}\label{lab59}\eta((P_{0,n}^A,P_{1,n}^A))\mathop{\longrightarrow}_{A\rightarrow\infty}0.\end{align}  This completes the claim and Step $1$.

For $n,\alpha\geq 1$, set $t_n^\alpha=\epsilon_ns_n^\alpha$ and $x_n^\alpha=\epsilon_ny_n^\alpha$ and
\begin{align*}
P_{0,n}^A(y)=\epsilon_n^\frac{d-2}{2}\rho_{0,n}^A(\epsilon_n y),\quad P_{1,n}^A(y)=\epsilon_n^{\frac{d-2}{2}+1}\rho_{1,n}^A(\epsilon_n y).
\end{align*}
Then ($\ref{lab54}$)-($\ref{lab56}$) imply the desired equalities ($\ref{lab48}$)-($\ref{lab49}$) and ($\ref{lab51}$).  To complete the proof of Lemma $\ref{lab46}$, it remains to show the limit ($\ref{lab50}$).  This is the content of the next step.

{\bf\underline{Step 2}}: ($\eta((P_{0,n}^A,P_{1,n}^A))\rightarrow 0$ implies $\displaystyle\limsup_{n\rightarrow\infty}\lVert S(t)(P_{0,n}^A,P_{1,n}^A)\rVert_{L_{t,x}^\frac{2(d+1)}{(d-2)}}\rightarrow 0$)

Observe that $(\nabla p_{0,n},p_{1,n})$ $(\epsilon_n)$-oscillatory implies \begin{align}
\label{lab59.25}\limsup_{n\rightarrow\infty} \int_{\{|\eta|\leq \frac{1}{R}\}\cup\{|\eta|\geq R\}} \left[|\eta|^2|\hat{P}_{0,n}(\eta)|^2+|\hat{P}_{1,n}(\eta)|^2\right]d\eta\mathop{\longrightarrow}_{R\rightarrow\infty}0.
\end{align}

Moreover, we claim that for $\alpha\geq 1$,
\begin{align}\nonumber \lefteqn{\limsup_{n\rightarrow\infty} \int_{\{|\eta|\leq \frac{1}{R}\}\cup\{|\eta|\geq R\}} \left[|\eta|^2|\left(S(-t_n^\alpha)(V_0^\alpha,V_1^\alpha)(\cdot -x_n^\alpha)\right)^\wedge(\eta)|^2+\right.}\\
&\label{lab59.5}\hspace{0.7in} \left.|\partial_t\left(S(-t_n^\alpha)(V_0^\alpha,V_1^\alpha)(\cdot -x_n^\alpha)\right)^\wedge(\eta)|^2\right]d\eta\mathop{\longrightarrow}_{R\rightarrow\infty} 0
\end{align}

To see ($\ref{lab59.5}$), for each $R>0$ define $\sigma_R(\xi)=\chi_{\{|\xi|\leq \frac{1}{R}\}\cup \{|\xi|\geq R\}}(\xi)$.  Then for $\alpha\geq 1$,
\begin{align*}
\sigma_R(D) S(s_n^{\alpha+1})(P^\alpha_{0,n},P^\alpha_{1,n})(y+y_n^{\alpha+1})\rightharpoonup \sigma_R(D)V_0^{\alpha+1}\quad\textrm{in}\quad\dot{H}^1
\end{align*}
 and
\begin{align*}
\sigma_R(D) \partial_t S(s_n^{\alpha+1})(P^\alpha_{0,n},P^\alpha_{1,n})(y+y_n^{\alpha+1})\rightharpoonup \sigma_R(D)V_1^{\alpha+1}\quad\textrm{in}\quad L^2.
\end{align*}
Note that $\sigma_R(D)S(t)(V_0^\alpha,V_1^\alpha)$ is a solution of the wave equation, and thus the conservation of energy and the weak lower semicontinuity of the norms imply that
\begin{align*}
&\lVert \sigma_R(D)S(-s_n^{\alpha+1})(V_0^{\alpha+1},V_1^{\alpha+1})(\cdot -y_n^{\alpha+1})\rVert^2_{\dot{H}^1}+\lVert \partial_t\sigma_R(D)S(-s_n^{\alpha+1})(V_0^{\alpha+1},V_1^{\alpha+1})(\cdot -y_n^{\alpha+1})\rVert^2_{L^2}\\
&=\lVert \sigma_R(D) V_0^{\alpha+1}\rVert^2_{\dot{H}^1}+\lVert \sigma_R(D)V_1^{\alpha+1}\rVert^2_{L^2}\\
&\hspace{0.3in}\leq \liminf_{n\rightarrow\infty} \lVert \sigma_R(D) S(s_n^{\alpha+1})(P_{0,n}^\alpha,P_{1,n}^\alpha)(y+y_n^{\alpha+1})\rVert^2_{\dot{H}^1}\\
&\hspace{0.6in}+\liminf_{n\rightarrow\infty}\lVert \partial_t\sigma_R(D)S(s_n^{\alpha+1})(P_{0,n}^\alpha,P_{1,n}^\alpha)(y+y_n^{\alpha+1})\rVert^2_{L^2}\\
&\hspace{0.9in}\leq \limsup_{n\rightarrow\infty} \left(\lVert \sigma_R(D) P_{0,n}^\alpha\rVert^2_{\dot{H}^1}+\lVert \sigma_R(D) P_{1,n}^\alpha\rVert^2_{L^2}\right).\\
\end{align*}
The limit ($\ref{lab59.5}$) then follows by induction.  Combining ($\ref{lab59.25}$) and ($\ref{lab59.5}$) and using the decompositions ($\ref{lab54}$)-($\ref{lab55}$) we get for every $A\geq 1$,
\begin{align}\label{lab60}\limsup_{n\rightarrow\infty} \int_{\{|\eta|\leq \frac{1}{R}\}\cup\{|\eta|\geq R\}} (|\eta|^2|\hat{P}_{0,n}^A(\eta)|^2+|\hat{P}_{1,n}^A(\eta)|^2)d\eta\mathop{\longrightarrow}_{R\rightarrow\infty}0.
\end{align}

We now claim that
\begin{align}
\nonumber\lefteqn{\limsup_{n\rightarrow\infty} \lVert S(t)(P_{0,n}^A,P_{1,n}^A)\rVert_{L_{t,x}^\frac{2(d+1)}{(d-2)}}}\hspace{1in}\\
&\leq C\limsup_{n\rightarrow\infty} \left[E((P_{0,n}^A,P_{1,n}^A))\right]^\frac{d-2}{2(d+1)}\eta((P_{0,n}^A,P_{1,n}^A))^\frac{6}{2(d+1)}.\label{lab61}
\end{align}
This estimate will be used along with the limit ($\ref{lab59}$) to show the smallness condition ($\ref{lab50}$)

Note that ($\ref{lab60}$) implies it is enough to prove the claim while assuming that for $a,b\in (0,\infty)$ with $a<b$, the supports of
$\widehat{P_{0,n}^A}$ and $\widehat{P_{1,n}^A}$ are contained in $\Omega:=\{\eta:a\leq |\eta|\leq b\}$ and the support of $[S(t)(P_{0,n},P_{1,n})]^{\sim}$ is contained in $\{(\sigma,\eta):a\leq |\sigma|\leq b, a\leq |\eta|\leq b\}$, where $\sim$ represents the Fourier transform in both the $t$ and $x$ variables.  The claim then follows immediately by an approximation argument.

We obtain the estimate ($\ref{lab61}$) in two steps.

$\bullet$ An estimate on $\displaystyle\limsup_{n\rightarrow\infty} \lVert S(t)(P_{0,n}^A,P_{1,n}^A)\rVert_{L_{t,x}^\frac{2(d+1)}{(d-2)}}$ which depends on $a$ and $b$:

Let us apply the interpolation inequality and ($\ref{lab22}$) to obtain the estimate,
\begin{align}
\nonumber\lVert S(t)(P_{0,n}^A,P_{1,n}^A)\rVert_{L^\frac{2(d+1)}{d-2}_{t,x}}&\leq\lVert S(t)(P_{0,n}^A,P_{1,n}^A)\rVert_{L_t^{\frac{2(d-1)}{d-2}}L_x^{\frac{2d(d-1)}{(d-2)^2}}}^\frac{d-1}{d+1}\lVert S(t)(P_{0,n}^A,P_{1,n}^A)\rVert_{L_t^\infty L_x^\frac{2d}{d-2}}^\frac{2}{d+1}\\
\label{lab62}&\leq C\left[E((P_{0,n}^A,P_{1,n}^A))\right]^\frac{d-1}{2(d+1)}\lVert S(t)(P_{0,n}^A,P_{1,n}^A)\rVert_{L_t^\infty L_x^\frac{2d}{d-2}}^\frac{2}{(d+1)}.
\end{align}

We next consider the norm $\lVert S(t)(P_{0,n}^A,P_{1,n}^A)\rVert_{L_t^\infty L_x^\frac{2d}{d-2}}$, where we apply the interpolation inequality to obtain
\begin{align}
\lVert S(t)(P_{0,n}^A,P_{1,n}^A)\rVert_{L_t^\infty L_x^\frac{2d}{d-2}}&\leq \lVert S(t)(P_{0,n}^A,P_{1,n}^A)\rVert_{L_t^\infty L_x^2}^\frac{d-2}{d}\lVert S(t)(P_{0,n}^A,P_{1,n}^A)\rVert_{L^\infty_{t,x}}^\frac{2}{d}.
\label{lab63}
\end{align}

We also have,
\begin{align}
\label{lab62.5}\lVert S(t)(P_{0,n}^A,P_{1,n}^A)\rVert_{L_t^\infty L_x^2}&\leq \sup_t \left(\int_\Omega \frac{|\eta|^2}{a^2}|(S(t)(P_{0,n}^A,P_{1,n}^A))^\wedge(\eta)|^2d\eta\right)^\frac{1}{2}\\
\nonumber&\leq\sup_t \frac{1}{a}\lVert \nabla S(t)(P_{0,n}^A,P_{1,n}^A)\rVert_{L_x^2(\mathbb{R}^d)}\\
&\leq\frac{C}{a}\left[E((P_{0,n}^A,P_{1,n}^A))\right]^\frac{1}{2},\label{lab64}
\end{align}
where to obtain ($\ref{lab62.5}$) we note that $a^2\leq |\eta|^2$ on $\Omega$ and the definition of $S(t)$ implies
\begin{align*}
(S(t)(P_{0,n}^A,P_{1,n}^A))^\wedge(\xi)&=\cos(t|\xi|)\widehat{P_{0,n}^A}(\xi)+\frac{1}{|\xi|}\sin(t|\xi|)\widehat{P_{1,n}^A}(\xi),
\end{align*}
so that $\supp [S(t)(P_{0,n}^A,P_{1,n}^A)]^{\wedge}\subset\Omega$ and $\supp \partial_s[S(t)(P_{0,n}^A,P_{1,n}^A)]^{\wedge}\subset\Omega$ for any $t\geq 0$.

We now estimate the norm $\lVert S(t)(P_{0,n}^A,P_{1,n}^A)\rVert_{L_{t,x}^\infty}$ in ($\ref{lab63}$).  Recall that $\tilde{f}$ denotes the Fourier transform of $f$ in both the $t$ and $x$ variables.  Choose $\chi\in \mathcal{S}(\mathbb{R}\times\mathbb{R}^d)$ such that $\widetilde{\chi}(\sigma,\eta)=1$ for each $(\sigma,\eta)\in\{a\leq |\sigma|\leq b,a\leq |\eta|\leq b\}$.  Then, for all $(\sigma,\eta)$, we have
\begin{align*}
[S(t)(P_{0,n}^A,P_{1,n}^A)]^{\sim}(\sigma,\eta)&=\widetilde{\chi}(\sigma,\eta)[S(t)(P_{0,n}^A,P_{1,n}^A)]^{\sim}(\sigma,\eta)\\
&=(\chi\star S(t)(P_{0,n}^A,P_{1,n}^A))^\sim (\sigma,\eta)\\
&=\left(\int\int \chi(-t,-x)S(\cdot+t)(P_{0,n}^A,P_{1,n}^A)(\cdot+x)dtdx\right)^{\sim}(\sigma,\eta).
\end{align*}

Taking the inverse Fourier transform of both sides, we see that
\begin{align*}
S(s)(P_{0,n}^A,P_{1,n}^A)(y)&=\int\int \chi(-t,-x)S(s+t)(P_{0,n}^A,P_{1,n}^A)(y+x)dtdx.
\end{align*}

We then use the definitions of $\limsup$, $\sup$, $\mathcal{V}$ and $\eta((P_{0,n}^A,P_{1,n}^A))$ to get
\begin{align}
\nonumber\lefteqn{\limsup_{n\rightarrow\infty} \lVert S(t)(P_{0,n}^A,P_{1,n}^A)\rVert_{L_{t,x}^\infty}}\hspace{1in}\\
\nonumber&=\sup_{\{s_n,y_n\}\subset \mathbb{R}\times\mathbb{R}^d} \limsup_{n\rightarrow\infty} |S(s_n)(P_{0,n}^A,P_{1,n}^A)(y_n)|\\
\nonumber&\leq \sup \left\{\left|\int\int \chi(-t,-x)S(t)(V_0,V_1)(x)dtdx\right|: (V_0,V_1)\in \mathcal{V}((P_{0,n}^A,P_{1,n}^A))\right\}\\
\nonumber &\leq\sup \left\{\lVert\chi\rVert_{L_t^1L_x^2}\lVert S(t)(V_0,V_1)\rVert_{L_t^\infty L_x^2}:(V_0,V_1)\in\mathcal{V}((P_{0,n}^A,P_{1,n}^A))\right\}\\
\nonumber &\leq C_{a,b}\sup \left\{E(V_0,V_1)^\frac{1}{2}:(V_0,V_1)\in\mathcal{V}((P_{0,n}^A,P_{1,n}^A))\right\}\\
\label{lab65}&\leq C_{a,b}\,\,\eta((P_{0,n}^A,P_{1,n}^A)).
\end{align}

Combining ($\ref{lab62}$)-($\ref{lab65}$), we have,
\begin{align}
\label{lab66}
\lefteqn{\limsup_{n\rightarrow\infty} \lVert S(t)(P_{0,n}^A,P_{1,n}^A)\rVert_{L^\frac{2(d+1)}{d-2}_{t,x}}}\\
&\hspace{0.75in}\leq C_{a,b}\limsup_{n\rightarrow\infty}\left[E((P_{0,n}^A,P_{1,n}^A))\right]^\frac{d^2+d-4}{2d(d+1)}\eta((P_{0,n}^A,P_{1,n}^A))^\frac{4}{d(d+1)}.\nonumber
\end{align}

$\bullet$ Removal of the dependence on $a$ and $b$:

To remove the dependence of the constant on $a$ and $b$, we first claim that
\begin{align}
\label{lab67}
\lVert S(t)(P_{0,n}^A,P_{1,n}^A)\rVert_{L^\frac{2(d+1)}{d-2}_{t,x}}^\frac{2(d+1)}{d-2}\mathop{\longrightarrow}_{n\rightarrow\infty} \sum_{\alpha=A+1}^\infty \lVert S(t)(V_0^\alpha,V_1^\alpha)\rVert_{L_{t,x}^\frac{2(d+1)}{d-2}}^\frac{2(d+1)}{d-2}.
\end{align}

To see this, note that ($\ref{lab54}$)-($\ref{lab55}$) and ($\ref{lab56}$) imply, for every $B>A$,
\begin{align}
\label{lab68}S(t)(P_{0,n}^A,P_{1,n}^A)&=\sum_{\alpha=A+1}^B S(s-s_n^\alpha)(V_0^\alpha,V_1^\alpha)(y-y_n^\alpha)+S(s)(P_{0,n}^B,P_{1,n}^B)(y),\\
\label{lab69}E((P_{0,n}^A,P_{1,n}^A))&=\sum_{\alpha=A+1}^B E((V_0^\alpha,V_1^\alpha))+E((P_{0,n}^B,P_{1,n}^B)),
\end{align}

By using ($\ref{lab68}$), Lemma \ref{lab31} and ($\ref{lab66}$), we have
\begin{align}
\nonumber \lefteqn{\limsup_{n\rightarrow\infty} \lVert S(t)(P_{0,n}^A,P_{1,n}^A)\rVert_{L_{t,x}^\frac{2(d+1)}{d-2}}}\hspace{.5in}\\
\nonumber &\leq \left(\sum_{\alpha=A+1}^B \lVert S(t)(V_0^\alpha,V_1^\alpha)\rVert_{L_{t,x}^\frac{2(d+1)}{d-2}}^\frac{2(d+1)}{d-2}\right)^\frac{1}{\frac{2(d+1)}{d-2}}+\limsup_{n\rightarrow\infty} \lVert S(t)(P_{0,n}^B,P_{1,n}^B)\rVert_{L_{t,x}^\frac{2(d+1)}{d-2}}\\
\nonumber &\leq\left(\sum_{\alpha=A+1}^B \lVert S(t)(V_0^\alpha,V_1^\alpha)\rVert_{L_{t,x}^\frac{2(d+1)}{d-2}}^\frac{2(d+1)}{d-2}\right)^\frac{d-2}{2(d+1)}+\\
\nonumber &\hspace{1in} C_{a,b}\limsup_{n\rightarrow\infty} \left[E((P_{0,n}^B,P_{1,n}^B))\right]^\frac{d^2+d-4}{2d(d+1)}\eta((P_{0,n}^B,P_{1,n}^B))^\frac{4}{d(d+1)}\\
\nonumber &\leq\left(\sum_{\alpha=A+1}^B \lVert S(t)(V_0^\alpha,V_1^\alpha)\rVert_{L_{t,x}^\frac{2(d+1)}{d-2}}^\frac{2(d+1)}{d-2}\right)^\frac{d-2}{2(d+1)}+\\
\label{lab69.5}&\hspace{1in} C_{a,b}\limsup_{n\rightarrow\infty} \left[E((P_{0,n},P_{1,n}))\right]^\frac{d^2+d-4}{2d(d+1)}\eta((P_{0,n}^B,P_{1,n}^B))^\frac{4}{d(d+1)},
\end{align}
where to obtain ($\ref{lab69.5}$) we have used ($\ref{lab69}$) and ($\ref{lab51}$).
Taking the limit in ($\ref{lab69.5}$) and using the limit ($\ref{lab59}$), we obtain,
\begin{align}
\label{lab70}
\limsup_{n\rightarrow\infty} \lVert S(t)(P_{0,n}^A,P_{1,n}^A)\rVert_{L_{t,x}^\frac{2(d+1)}{d-2}}&\leq \left(\sum_{\alpha=A+1}^\infty \lVert S(t)(V_0^\alpha,V_1^\alpha)\rVert_{L_{t,x}^\frac{2(d+1)}{d-2}}^\frac{2(d+1)}{d-2}\right)^\frac{d-2}{2(d+1)}.
\end{align}

On the other hand, ($\ref{lab68}$) and Lemma $\ref{lab31}$ imply,
\begin{align*}
\lefteqn{\left(\sum_{\alpha=A+1}^B \lVert S(t)(V_0^\alpha,V_1^\alpha)\rVert_{L_{t,x}^\frac{2(d+1)}{d-2}}^\frac{2(d+1)}{d-2}\right)^\frac{d-2}{2(d+1)}}&\\
&\hspace{0.5in}=\liminf_{n\rightarrow\infty} \lVert \sum_{\alpha=A+1}^B S(s-s_n^\alpha)(V_0^\alpha,V_1^\alpha)(y-y_n^\alpha)\rVert_{L_{t,x}^\frac{2(d+1)}{d-2}}\\
&\hspace{0.5in}\leq \liminf_{n\rightarrow\infty} \lVert S(t)(P_{0,n}^A,P_{1,n}^A)\rVert_{L_{t,x}^\frac{2(d+1)}{d-2}}+\liminf_{n\rightarrow\infty} \lVert S(t)(P_{0,n}^B,P_{1,n}^B)\rVert_{L_{t,x}^\frac{2(d+1)}{d-2}}\\
&\hspace{0.5in}\leq\liminf_{n\rightarrow\infty} \lVert S(t)(P_{0,n}^A,P_{1,n}^A)\rVert_{L_{t,x}^\frac{2(d+1)}{d-2}}+\\&\hspace{1in}C_{a,b}\limsup_{n\rightarrow\infty} \left[E((P_{0,n},P_{1,n}))\right]^\frac{d^2+d-4}{2d(d+1)}\eta((P_{0,n}^B,P_{1,n}^B))^\frac{4}{d(d+1)},
\end{align*}
where to obtain the last inequality we have used ($\ref{lab66}$), ($\ref{lab69}$) and ($\ref{lab51}$).
Using ($\ref{lab59}$), we get,
\begin{align}
\label{lab71}
\left(\sum_{\alpha=A+1}^\infty \lVert S(t)(V_0^\alpha,V_1^\alpha)\rVert_{L^\frac{2(d+1)}{d-2}}^\frac{2(d+1)}{d-2}\right)^\frac{d-2}{2(d+1)}&\leq\liminf_{n\rightarrow\infty} \lVert S(t)(P_{0,n}^A,P_{1,n}^A)\rVert_{L_{t,x}^\frac{2(d+1)}{d-2}}.
\end{align}

Combining the estimates ($\ref{lab70}$) and ($\ref{lab71}$) we obtain the claimed limit ($\ref{lab67}$).

We now return to the removal of the dependence of the estimate ($\ref{lab66}$) on $a$ and $b$.  Using ($\ref{lab67}$) with Strichartz's inequality, we get,
\begin{align}
\nonumber\lefteqn{\limsup_{n\rightarrow\infty} \lVert S(t)(P_{0,n}^A,P_{1,n}^A)\rVert_{L_{t,x}^\frac{2(d+1)}{(d-2)}}}\hspace{1in}\\
\nonumber&=\left(\sum_{\alpha=A+1}^\infty \lVert S(t)(V_0^\alpha,V_1^{\alpha})\rVert_{L_{t,x}^\frac{2(d+1)}{d-2}}^\frac{2(d+1)}{d-2}\right)^\frac{d-2}{2(d+1)}\\
\nonumber&\leq\left[\left(\sum_{\alpha=A+1}^\infty \lVert S(t)(V_0^\alpha,V_1^\alpha)\rVert_{L_{t,x}^\frac{2(d+1)}{d-2}}^2\right)\sup_{\alpha> A} \lVert S(t)(V_0^\alpha,V_1^{\alpha})\rVert_{L_{t,x}^\frac{2(d+1)}{d-2}}^\frac{6}{d-2}\right]^\frac{d-2}{2(d+1)}\\
\nonumber&\leq\left[C\left(\sum_{\alpha=A+1}^\infty \left[E((V_0^\alpha, V_1^\alpha))\right]\right)\limsup_{n\rightarrow\infty}\eta((P_{0,n}^A,P_{1,n}^A))^\frac{6}{d-2}\right]^\frac{d-2}{2(d+1)}\\
\label{lab72}&\leq\left[C\limsup_{n\rightarrow\infty} \left[E((P_{0,n}^A,P_{1,n}^A))\right]\,\,\eta((P_{0,n}^A,P_{1,n}^A))^\frac{6}{d-2}\right]^\frac{d-2}{2(d+1)},
\end{align}
where we use ($\ref{lab69}$) for the last inequality.  This gives the claimed inequality ($\ref{lab61}$).

To obtain ($\ref{lab50}$), observe that ($\ref{lab56}$) implies,
\begin{align}
\label{lab73}\limsup_{n\rightarrow\infty} E((P_{0,n}^A,P_{1,n}^A))&\leq\limsup_{n\rightarrow\infty}E((P_{0,n},P_{1,n})),
\end{align}
where the term on the right hand side is bounded due to the boundedness of the sequence $(P_{0,n},P_{1,n})$ in $\dot{H}^1\times L^2$.  Combining ($\ref{lab72}$) and ($\ref{lab73}$), we get,
\begin{align*}
\limsup_{n\rightarrow\infty} \lVert S(t)(P_{0,n}^A,P_{1,n}^A)\rVert_{L_{t,x}^\frac{2(d+1)}{(d-2)}}&\leq C\eta((P_{0,n}^A,P_{1,n}^A))^\frac{6}{2(d+1)}
\end{align*}
which, along with ($\ref{lab59}$), completes the proof of Lemma $\ref{lab46}$.
\end{proof}
\paragraph{\bf Synthesis.}\ \\
\ \\
We now turn to the claims stated in Theorem $\ref{lab4}$.

Applying Lemma $\ref{lab47}$ to each $(p_{0,n}^j,p_{1,n}^j)_{n\in\mathbb{N}}$, using a diagonal extraction and passing to a subsequence, one obtains, \begin{align} p_{0,n}^j(x)&=\sum_{\alpha=1}^{A_j} V_n^{j,\alpha}\left(0,x\right)+\rho_{0,n}^{j,A_j}(x),\\
p_{1,n}^j(x)&=\sum_{\alpha=1}^{A_j} \partial_t V_n^{j,\alpha}(0,x)+\rho_{1,n}^{j,A_j}(x).\end{align}

Now, plugging the above equalities into ($\ref{lab42}$), we have, \begin{align}\label{lab74}u_{0,n}(x)&=\sum_{j=1}^l \left(\sum_{\alpha=1}^{A_j} V_n^{j,\alpha}(0,x)\right)+w_{0,n}^{l,A_1,\cdots,A_l}(x),\\
u_{1,n}(x)&=\sum_{j=1}^l \left(\sum_{\alpha=1}^{A_j} \partial_t V_n^{j,\alpha}(0,x)\right)+w_{1,n}^{l,A_1,\cdots,A_l}(x),\end{align}
where \begin{align} \label{lab75}w_{0,n}^{l,A_1,\cdots ,A_l}(x)&=r_n^l(x)+\sum_{j=1}^l \rho_{0,n}^{j,A_j}(x),\\
w_{1,n}^{\alpha,A_1,\cdots,A_l}(x)&=s_n^l(x)+\sum_{j=1}^l \rho_{1,n}^{j,A_j}(x).\end{align}

Moreover by ($\ref{lab51}$), \begin{align} \nonumber \lefteqn{\lVert u_{0,n}\rVert_{\dot{H}^1}^2+\lVert u_{1,n}\rVert_{L^2}^2}&\\
\nonumber &\hspace{0.2in}=\sum_{j=1}^l \left(\lVert p_{0,n}^j\rVert_{\dot{H}^1}^2+\lVert p_{1,n}^j\rVert_{L^2}^2\right)+\lVert r_n^l\rVert_{\dot{H}^1}^2+\lVert s_n^l\rVert_{L^2}^2+o(1)\\
\nonumber &\hspace{0.2in}=\sum_{j=1}^l\left(\sum_{\alpha=1}^{A_j} \left(\lVert V_0^{j,\alpha}\rVert_{\dot{H}^1}+\lVert V_1^{j,\alpha}\rVert_{L^2}^2\right)+\lVert \rho_{0,n}^{j,A_j}\rVert_{\dot{H}^1}^2+\lVert \rho_{1,n}^{j,A_j}\rVert_{L^2}^2\right)+\lVert r_n^l\rVert_{\dot{H}^1}^2+\lVert s_n^l\rVert_{L^2}^2+o(1)\\
\label{lab76}&\hspace{0.2in}=\sum_{j=1}^l\left(\sum_{\alpha=1}^{A_j} \lVert V_0^{j,\alpha}\rVert_{\dot{H}^1}^2+\lVert V_1^{j,\alpha}\rVert_{L^2}^2\right)+\lVert w_{0,n}^{l,A_1,\cdots,A_l}\rVert_{\dot{H}^1}^2+\lVert w_{1,n}^{l,A_1,\cdots,A_l}\rVert_{L^2}^2+o(1),
\end{align}
where we use ($\ref{lab75}$) and the identity ($\ref{lab28}$) for the last equality.

We consider an enumeration $m$ of the pairs $(j,\alpha)$ with \begin{align*} m(j,\alpha)<m(k,\beta)\end{align*} for $j+\alpha< k+\beta$.  The above discussion then gives the claims ($\ref{lab6}$)-($\ref{lab7}$) and ($\ref{lab9}$) of Theorem $\ref{lab4}$.

To see the claim ($\ref{lab5}$), let $j\neq j'$ be given with $j=m(i,\alpha),j'=m(i',\alpha')$.  If $i\neq i'$, then ($\ref{lab29}$) implies that ($\ref{lab5}$) holds.  On the other hand, if we are in the case $i=i'$, then ($\ref{lab47.5}$) also implies ($\ref{lab5}$).  Thus in either case, ($\ref{lab5}$) holds, which was the desired claim.

To complete the theorem, it remains to prove the smallness condition ($\ref{lab8}$) and the limit ($\ref{lab11}$).  For the proof of ($\ref{lab8}$), it suffices to show \begin{align}\label{lab78}\lVert S(t)(w_{0,n}^{l,A_1,\cdots, A_l},w_{1,n}^{l,A_1,\cdots,A_l})\rVert_{L_{t,x}^\frac{2(d+1)}{d-2}}\mathop{\longrightarrow}_{n\rightarrow\infty} 0,\end{align}
when $\inf\{l,j+A_j:1\leq j\leq l\}$ tends to $\infty$.

Indeed, note that for each $\epsilon>0$, we may use the limit ($\ref{lab45}$) to choose $l_0$ such that for $l\geq l_0$, \begin{align}\label{lab79} \limsup_{n\rightarrow\infty} \lVert  S(t)(r_n^l,s_n^l)\rVert_{L_{t,x}^\frac{2(d+1)}{d-2}}&\leq\frac{\epsilon}{3}.\end{align}

Then, for every $l\geq l_0$, using the limit ($\ref{lab50}$), we choose $B_l$ such that for each $A\geq B_l$ and each $j\in \{1,\cdots,l\}$, \begin{align}\label{lab80} \limsup_{n\rightarrow\infty} \lVert S(t)(\rho_{0,n}^{j,A},\rho_{1,n}^{j,A})\rVert_{L_{t,x}^\frac{2(d+1)}{d-2}}&\leq\frac{\epsilon}{3l}.\end{align}

One then writes $S(t)(w_{0,n}^{l,A_1,\cdots,A_l},w_{1,n}^{l,A_1,\cdots,A_l})$ in the form \begin{align*} \lefteqn{S(t)(w_{0,n}^{l,A_1,\cdots,A_l},w_{1,n}^{l,A_1,\cdots,A_l})}\\
&=S(t)(r_n^l,s_n^l)+\sum_{\stackrel{1\leq j\leq l}{A_j<B_l}} S(t)(\rho_{0,n}^{j,B_l},\rho_{1,n}^{j,B_l})+\sum_{\stackrel{1\leq j\leq l}{A_j\geq B_l}} S(t)(\rho_{0,n}^{j,A_j},\rho_{1,n}^{j,A_j})+\gamma_n^{l,A_1,\cdots,A_l}\\
&=S(t)(r_n^l,s_n^l)+\sum_{1\leq j\leq l} S(t)(\rho_{0,n}^{j,\max\{A_j,B_l\}},\rho_{1,n}^{j,\max\{A_j,B_l\}})+\gamma_n^{l,A_1,\cdots,A_l} \end{align*}
where
\begin{align*}\gamma_n^{l,A_1,\cdots,A_l}&=\sum_{\stackrel{1\leq j\leq l}{A_j<B_l}}S(t)(\rho_{0,n}^{j,A_j},\rho_{1,n}^{j,A_j})-S(t)(\rho_{0,n}^{j,B_l},\rho_{1,n}^{j,B_l}),\\
&=\sum_{\stackrel{1\leq j\leq l}{A_j<B_l}}\sum_{A_j<\alpha\leq B_l} V_n^{j,\alpha}.\end{align*}

Applying ($\ref{lab79}$) and ($\ref{lab80}$), we have \begin{align}\label{lab81}\limsup_{n\rightarrow\infty} \lVert S(t)(w_{0,n}^{l,A_1,\cdots,A_l},w_{1,n}^{l,A_1,\cdots,A_l})\rVert_{L_{t,x}^\frac{2(d+1)}{d-2}}&\leq\frac{2\epsilon}{3}+\limsup_{n\rightarrow\infty} \lVert \gamma_n^{l,A_1,\cdots,A_l}\rVert_{L_{t,x}^\frac{2(d+1)}{d-2}}.\end{align}

Then ($\ref{lab5}$) allows us to use Lemma $\ref{lab31}$ to see \begin{align}\label{lab82} \lVert \gamma_n^{l,A_1,\cdots,A_l}\rVert_{L_{t,x}^\frac{2(d+1)}{d-2}}^\frac{2(d+1)}{d-2}&=\sum_{\stackrel{1\leq j\leq l}{A_j<B_l}}\,\,\sum_{A_j<\alpha\leq B_l} \lVert V^{j,\alpha}\rVert_{L_{t,x}^\frac{2(d+1)}{d-2}}^\frac{2(d+1)}{d-2}+o(1),\quad n\rightarrow\infty.\end{align}

Moreover, by ($\ref{lab22}$) it follows that \begin{align} \sum_{(j,\alpha)} \lVert V^{j,\alpha}\rVert_{L_{t,x}^\frac{2(d+1)}{d-2}}^\frac{2(d+1)}{d-2}\leq \sum_{(j,\alpha)} E((V_0^{j,\alpha},V_1^{j,\alpha}))^\frac{d+1}{d-2}.\label{lab83}\end{align}

Notice that ($\ref{lab76}$) implies $\displaystyle \sum_{(j,\alpha)} E((V_0^{j,\alpha},V_1^{j,\alpha}))$ converges, so that the right hand side in ($\ref{lab83}$) is finite.  Hence,
\begin{align}
\sum_{\stackrel{(j,\alpha)}{A_j<B_l}}\sum_{A_j<\alpha\leq B_l} \lVert V^{j,\alpha}\rVert_{L_{t,x}^\frac{2(d+1)}{d-2}}^\frac{2(d+1)}{d-2}\leq \sum_{\stackrel{(j,\alpha)}{\alpha >A_j,1\leq j\leq l}}\lVert V^{j,\alpha}\rVert_{L_{t,x}^\frac{2(d+1)}{d-2}}^\frac{2(d+1)}{d-2}\leq\left(\frac{\epsilon}{3}\right)^\frac{2(d+1)}{d-2}\label{lab84}\end{align}
for $\inf\{j+A_j:1\leq j\leq l\}$ large enough.

Then ($\ref{lab81}$), ($\ref{lab82}$) and ($\ref{lab84}$) give \begin{align} \limsup_{n\rightarrow\infty} \lVert S(t)(w_{0,n}^{l,A_1,\cdots,A_l},w_{1,n}^{l,A_1,\cdots,A_l})\rVert_{L_{t,x}^\frac{2(d+1)}{d-2}}\leq\epsilon,
\end{align}
for $\inf\{ j+A_j,1\leq j\leq l\}$ large enough.  Since $\epsilon>0$ is arbitrary, the desired conclusion ($\ref{lab78}$) follows.

Finally, the claim ($\ref{lab11}$) follows from an argument identical to the proof of Lemma $\ref{lab31}$ and will be omitted.  This completes the proof of Theorem $\ref{lab4}$.{\hfill\ensuremath{\Box}}

\begin{remark}
The smallness condition ($\ref{lab8}$) in Theorem $\ref{lab4}$, which involves the diagonal pair, can be generalized to any suitable wave admissible pair $(q,r)$.  We will use this fact in proving the existence of maximizers in the second part of this section and in Section 4.
\end{remark}
\begin{corollary}
\label{lab85}
Suppose the hypotheses of Theorem $\ref{lab4}$ hold.  Let $(q,r)$ be a wave admissible pair with $q,r\in (2,\infty)$ and satisfying the $\dot{H}^1$-scaling condition.  Then \begin{align*}\limsup_{n\rightarrow\infty} \lVert S(t)(w_{0,n}^l,w_{1,n}^l)\rVert_{L_t^qL_x^r}\mathop{\longrightarrow}_{l\rightarrow\infty}0.\end{align*}
\end{corollary}
\begin{proof}
We consider two cases, based on the size of $q$.  Note that the case $q=\frac{2(d+1)}{d-2}$ was settled in Theorem $\ref{lab4}$.

\underline{\bf Case $1$}:  $q>\frac{2(d+1)}{d-2}$.

Note that $(\infty,\frac{2d}{d-2})$ is a wave admissible pair satisfying the $\dot{H}^1$-scaling condition.  Choose $\alpha=1-\frac{2(d+1)}{q(d-2)}$, so that $0<\alpha<1$ and $\frac{1}{q}=\frac{\alpha}{\infty}+\frac{(1-\alpha)(d-2)}{2(d+1)}$, $\frac{1}{r}=\frac{\alpha (d-2)}{2d}+\frac{(1-\alpha)(d-2)}{2(d+1)}$.

The interpolation inequality then shows that for every $l\geq 1$,
\begin{align*}
\limsup_{n\rightarrow\infty}\lVert S(t)(w_{0,n}^l,w_{1,n}^l)\rVert_{L_t^qL_x^r}&\leq\limsup_{n\rightarrow\infty} \left(\lVert S(t)(w_{0,n}^l,w_{1,n}^l)\rVert_{L_t^\infty L_x^\frac{2d}{d-2}}^\alpha\lVert S(t)(w_{0,n}^l,w_{1,n}^l)\rVert_{L_{t,x}^\frac{2(d+1)}{d-2}}^{(1-\alpha)}\right)\\
&\leq\limsup_{n\rightarrow\infty} \left(\left[E((w_{0,n}^l,w_{1,n}^l))\right]^\frac{\alpha}{2}\lVert S(t)(w_{0,n}^l,w_{1,n}^l)\rVert_{L_{t,x}^\frac{2(d+1)}{d-2}}^{(1-\alpha)}\right)\\
&\leq\left[E((u_{0,n},u_{1,n}))\right]^\frac{\alpha}{2}\limsup_{n\rightarrow\infty} \lVert S(t)(w_{0,n}^l,w_{1,n}^l)\rVert_{L_{t,x}^\frac{2(d+1)}{d-2}}^{(1-\alpha)},
\end{align*}
where we use ($\ref{lab22}$) and ($\ref{lab9}$).  Thus, we may let $l$ tend to $\infty$ to see that,
\begin{align*}
\limsup_{n\rightarrow\infty} \lVert S(t)(w_{0,n}^l,w_{1,n}^l)\rVert_{L_t^qL_x^r}\mathop{\longrightarrow}_{l\rightarrow\infty} 0.
\end{align*}

\underline{\bf Case 2}: $q<\frac{2(d+1)}{d-2}$.

Define $q_1=\frac{q+2}{2}$, $r_1=\frac{2d(q+2)}{dq-2q+2d-8}$, and note that $(q_1,r_1)$ is a wave admissible pair with $r_1<\infty$ and satisfying the $\dot{H}^1$-scaling condition.  Thus, choosing $\alpha=\frac{(dq-2q-2d-2)(q+2)}{q(dq-2q-2d-8)}$, the interpolation inequality shows that,
\begin{align*}
\limsup_{n\rightarrow\infty} \lVert S(t)(w_{0,n}^l,w_{1,n}^l)\rVert_{L_t^qL_x^r}&\leq\limsup_{n\rightarrow\infty} \left(\lVert S(t)(w_{0,n}^l,w_{1,n}^l)\rVert_{L_t^{q_1}L_x^{r_1}}^\alpha\lVert S(t)(w_{0,n}^l,w_{1,n}^l)\rVert_{L_{t,x}^\frac{2(d+1)}{d-2}}^{(1-\alpha)}\right)\\
&\leq\limsup_{n\rightarrow\infty} \left(\left[E((w_{0,n}^l,w_{1,n}^l))\right]^\frac{\alpha}{2}\lVert S(t)(w_{0,n}^l,w_{1,n}^l)\rVert_{L_{t,x}^\frac{2(d+1)}{d-2}}^{(1-\alpha)}\right)\\
&\leq \left[E((u_{0,n},u_{1,n}))\right]^\frac{\alpha}{2}\limsup_{n\rightarrow\infty} \lVert S(t)(w_{0,n}^l,w_{1,n}^l)\rVert_{L_{t,x}^\frac{2(d+1)}{d-2}}^{(1-\alpha)},
\end{align*}
where once more we use ($\ref{lab22}$) and ($\ref{lab9}$) in the last inequality.  Thus, we may let $l$ tend to $\infty$ to see that,
\begin{align*}
\limsup_{n\rightarrow\infty} \lVert S(t)(w_{0,n}^l,w_{1,n}^l)\rVert_{L_t^qL_x^r}\mathop{\longrightarrow}_{l\rightarrow \infty}0.
\end{align*}
\end{proof}

\subsection{Existence of Maximizers}\ \\
\ \\
We start with a preliminary lemma which will be a useful tool in proving the existence of maximizers, Theorem $\ref{lab12}$.
\begin{lemma}
\label{lab86}
Suppose the hypotheses of Theorem $\ref{lab4}$ hold.  Let $(q,r)$ be a wave admissible pair with $q,r\in (2,\infty)$ and satisfying the $\dot{H}^1$-scaling condition.  Then, if $B=\min\{q,r\}$, for every $l\geq 1$, \begin{align*}\limsup_{n\rightarrow\infty} \left\lVert \sum_{j=1}^l V_n^j\right\rVert_{L_t^qL_x^r}^B\leq \sum_{j=1}^l \left\lVert V^j\right\rVert_{L_t^qL_x^r}^B.\end{align*}
\end{lemma}
\begin{proof}
The idea of this proof is based on an orthogonality property that was used to prove an analogous statement for the Schr\"odinger equation; see Lemma $1.6$ in \cite{ShaoSchrodinger}, as well as Lemma $5.5$ in \cite{BegoutVargas} and a similar discussion in \cite{Gerard}.

We first claim that ($\ref{lab11}$) implies, for $j\neq k$, the limit \begin{align}\label{lab87}\lVert V_n^jV_n^k\rVert_{L_t^{q/2}L_x^{r/2}}\mathop{\longrightarrow}_{n\rightarrow\infty} 0.\end{align}  We will use this fact to complete the proof of the lemma.  The argument to prove ($\ref{lab87}$) is similar to the proof given in Corollary $\ref{lab85}$ with suitable modifications and the extra ingredient of H\"older's inequality, and therefore will be omitted.

Note that for $f,g\in L_t^{q/r}L_x^1$, we have \begin{align*}\lVert f+g\rVert_{L_t^{q/r}L_x^1}^{B/r}\leq \lVert f\rVert_{L_t^{q/r}L_x^1}^{B/r}+\lVert g\rVert_{L_t^{q/r}L_x^1}^{B/r}.\end{align*}

Thus,
\begin{align*}
\lefteqn{\left\lVert \sum_{j=1}^l V_n^j\right\rVert_{L_t^qL_x^r}^B=\left\lVert\,\left|\sum_{j=1}^l V_n^j\right|^r\,\right\rVert_{L_t^{q/r}L_x^1}^{B/r}=\left\lVert\left|\sum_{j=1}^l V_n^j\right|\left|\sum_{i=1}^l V_n^i\right|\left|\sum_{k=1}^l V_n^k\right|^{r-2}\right\rVert_{L_t^{q/r}L_x^1}^{B/r}}\\
&\leq\left\lVert \sum_{j=1}^l \left[ |V_n^j|\left(\sum_{i=1}^l |V_n^i|\right)\left|\sum_{k=1}^l V_n^k\right|^{r-2}\right]\right\rVert_{L_t^{q/r}L_x^1}^{B/r}\\
&\leq\sum_{j=1}^l \left[\left\lVert |V_n^j|^2\left|\sum_{k=1}^l V_n^k\right|^{r-2}\right\rVert_{L_t^{q/r}L_x^1}^{B/r}+\left\lVert \sum_{i\neq j, i\leq l} |V_n^j||V_n^i|\left|\sum_{k=1}^l V_n^k\right|^{r-2}\right\rVert_{L_t^{q/r}L_x^1}^{B/r}\right]\\
&=:\sum_{j=1}^l (I)_j+(II)_j.
\end{align*}

Define $s=\lfloor r-2\rfloor$, the greatest integer less than $r-2$.  Then $0\leq r-2-s<1$, so that, for $j=1,...,l$,
\begin{align*}
(I)_j&\leq\left\lVert |V_n^j|^2\left(\sum_{i=1}^l |V_n^i|\right)^{r-2-s}\left|\sum_{k=1}^l V_n^k\right|^{s}\right\rVert_{L_t^{q/r}L_x^1}^{B/r}\\
&\leq\left\lVert |V_n^j|^2\left(\sum_{i=1}^l |V_n^i|^{r-2-s}\right)\left|\sum_{k=1}^l V_n^k\right|^s\right\rVert_{L_t^{q/r}L_x^1}^{B/r}\\
&\leq\left\lVert |V_n^j|^{r-s}\left(\sum_{k=1}^l |V_n^k|\right)^s\right\rVert_{L_t^{q/r}L_x^1}^{B/r}+\sum_{i\leq l: i\neq j} \left\lVert |V_n^j|^{4+s-r}|V_n^jV_n^i|^{r-2-s}\left|\sum_{k=1}^l V_n^k\right|^s\right\rVert_{L_t^{q/r}L_x^1}^{B/r}\\
&=\left\lVert |V_n^j|^{r-s}\left(\sum_{k=1}^l |V_n^k|^s+\sum_{(k_1,k_2,\cdots,k_s)}|V_n^{k_1}V_n^{k_2}||V_n^{k_3}|\cdots|V_n^{k_s}|\right)\right\rVert_{L_t^{q/r}L_x^1}^{B/r}+\\
&\hspace{.75in}\sum_{i\leq l: i\neq j} \left\lVert |V_n^j|^{4+s-r}|V_n^jV_n^i|^{r-2-s}\left|\sum_{k=1}^l V_n^k\right|^s\right\rVert_{L_t^{q/r}L_x^1}^{B/r}\\
&\leq\sum_{k=1}^l\left\lVert |V_n^j|^{r-s}|V_n^k|^s\right\rVert_{L_t^{q/r}L_x^1}^{B/r}+\\
&\hspace{.60in}\sum_{(k_1,k_2,\cdots,k_s)}\left\lVert|V_n^j|^{r-s}|V_n^{k_1}V_n^{k_2}||V_n^{k_3}|\cdots|V_n^{k_s}|\right\rVert_{L_t^{q/r}L_x^1}^{B/r}+\\
&\hspace{.75in}\sum_{i\leq l: i\neq j} \left\lVert |V_n^j|^{4+s-r}|V_n^jV_n^i|^{r-2-s}\left|\sum_{k=1}^l V_n^k\right|^s\right\rVert_{L_t^{q/r}L_x^1}^{B/r}\\
&=\lVert |V_n^j|^r\rVert_{L_t^{q/r}L_x^1}^{B/r}+\left(\sum_{k\leq l:k\neq j} (Ia)_{j,k}\right)+\\
&\hspace{.60in}\left(\sum_{(k_1,k_2,\cdots,k_s)}(Ib)_{j,k_1,k_2,\cdots,k_s}\right)+\left(\sum_{i\leq l: i\neq j}(Ic)_{j,i}\right),
\end{align*}
where the sum of the terms $(Ib)$ is over $s$-tuples $(k_1,k_2,\cdots,k_s)$ such that at least two $k_i$s are different (without loss of generality, we assume $k_1\neq k_2$) and for each $i,j,k\in \{1,\cdots,l\}$ and $s$-tuple $(k_1,\cdots,k_s)$,
\begin{align*}
(Ia)_{j,k}&=\lVert |V_n^j|^{r-s}|V_n^k|^s\rVert_{L_t^{q/r}L_x^1}^{B/r},\\
(Ib)_{j,k_1,k_2,\cdots, k_s}&=\lVert |V_n^j|^{r-s}|V_n^{k_1}V_n^{k_2}||V_n^{k_3}|\cdots|V_n^{k_s}|\rVert_{L_t^{q/r}L_x^1}^{B/r},\\
(Ic)_{j,i}&=\left\lVert |V_n^j|^{4+s-r}|V_n^jV_n^i|^{r-2-s}\left|\sum_{k=1}^l V_n^k\right|^s\right\rVert_{L_t^{q/r}L_x^1}^{B/r}.
\end{align*}
Note that for each $j=1,\cdots, l$, $k\neq j$, if $r-s\leq s$, then
\begin{align}
\nonumber (Ia)_{j,k}&=\lVert |V_n^j|^{r-s}|V_n^k|^{r-s}|V_n^k|^{2s-r}\rVert_{L_t^{q/r}L_x^1}^{B/r}\\
\nonumber &\leq\lVert |V_n^jV_n^k|^{r-s}\rVert_{L_t^\frac{q}{2(r-s)}L_x^\frac{r}{2(r-s)}}^{B/r}\lVert |V_n^k|^{2s-r}\rVert_{L_t^\frac{q}{2s-r}L_x^\frac{r}{2s-r}}^{B/r}\\
\label{lab88}&=\lVert V_n^jV_n^k\rVert_{L_t^{q/2}L_x^{r/2}}^{B(r-s)/r}\lVert V_n^k\rVert_{L_t^qL_x^r}^{B(2s-r)/r}.
\end{align}

If $r-s>s$, then
\begin{align}
\nonumber (Ia)_{j,k}&=\lVert |V_n^j|^{s}|V_n^k|^{s}|V_n^j|^{r-2s}\rVert_{L_t^{q/r}L_x^1}^{B/r}\\
\nonumber &\leq\lVert |V_n^jV_n^k|^{s}\rVert_{L_t^\frac{q}{2s}L_x^\frac{r}{2s}}^{B/r}\lVert |V_n^j|^{r-2s}\rVert_{L_t^\frac{q}{r-2s}L_x^\frac{r}{r-2s}}^{B/r}\\
\label{lab89}&=\lVert V_n^jV_n^k\rVert_{L_t^{q/2}L_x^{r/2}}^{Bs/r}\lVert V_n^j\rVert_{L_t^qL_x^r}^{B(r-2s)/r}.
\end{align}

To estimate the terms of the form $(Ib)_{j,k_1,\cdots,k_s}$, note that for each $j=1,\cdots, l$, $(k_1,\cdots,k_s)$ with $k_1\neq k_2$,
\begin{align}
\label{lab90}(Ib)_{j,k_1,\cdots, k_s}&\leq\left(\lVert V_n^{k_1}V_n^{k_2}\rVert_{L_t^{q/2}L_x^{r/2}}\lVert |V_n^j|^{r-s}\rVert_{L_t^\frac{q}{r-s}L_x^\frac{r}{r-s}}\lVert V_n^{k_3}\rVert_{L_t^qL_x^r}\cdots\lVert V_n^{k_s}\rVert_{L_t^qL_x^r}\right)^{B/r}
\end{align}
since $\frac{r}{q}=\frac{2}{q}+\frac{r-s}{q}+(s-2)\frac{1}{q}$ and $\frac{1}{1}=\frac{2}{r}+\frac{r-s}{r}+(s-2)\frac{1}{r}$.

Note that for $j=1,\cdots,l$, $i\neq j$,
\begin{align}
\nonumber (Ic)_{j,i}&\leq\lVert |V_n^jV_n^i|^{r-2-s}\rVert_{L_t^\frac{q}{2(r-2-s)}L_x^\frac{r}{2(r-2-s)}}^{B/r}\lVert |V_n^j|^{4+s-r}\left|\sum_{k=1}^l V_n^k\right|^s\rVert_{L_t^\frac{q}{4+2s-r}L_x^\frac{r}{4+2s-r}}^{B/r}\\
\nonumber &\leq\lVert V_n^jV_n^i\rVert_{L_t^{q/2}L_x^{r/2}}^{B(r-2-s)/r}\lVert |V_n^j|^{4+s-r}\rVert_{L_t^\frac{q}{4+s-r}L_x^\frac{r}{4+s-r}}^{B/r}\lVert\left|\sum_{k=1}^l V_n^k\right|^s\rVert_{L_t^\frac{q}{s}L_x^\frac{r}{s}}^{B/r}\\
\label{lab91}&=\lVert V_n^jV_n^i\rVert_{L_t^{q/2}L_x^{r/2}}^{B(r-2-s)/r}\lVert V_n^j\rVert_{L_t^qL_x^r}^{B(4+s-r)/r}\lVert\sum_{k=1}^l V_n^k\rVert_{L_t^qL_x^r}^{Bs/r}
\end{align}

Note that for $j=1,\cdots,l$, we have
\begin{align}
\nonumber (II)_j&\leq\sum_{i\neq j, i\leq l} \left\lVert |V_n^j||V_n^i|\left|\sum_{k=1}^l V_n^k\right|^{r-2}\right\rVert_{L_t^{q/r}L_x^1}^{B/r}\\
\nonumber &\leq\sum_{i\neq j, i\leq l} \left\lVert V_n^jV_n^i\right\rVert_{L_t^{q/2}L_x^{r/2}}^{B/r}\left\lVert \left(\sum_{k=1}^l V_n^k\right)^{r-2}\right\rVert_{L_t^{\frac{q}{r-2}}L_x^\frac{r}{r-2}}^{B/r}\\
&=\sum_{i\neq j, i\leq l} \left\lVert V_n^jV_n^i\right\rVert_{L_t^{q/2}L_x^{r/2}}^{B/r}\left\lVert \sum_{k=1}^l V_n^k\right\rVert_{L_t^qL_x^r}^{B(r-2)/r}\label{lab92}
\end{align}

We now take the limit as $n$ tends to infinity in the above inequalities ($\ref{lab88}$)-($\ref{lab92}$).  As we have noted in ($\ref{lab87}$), for each $j\neq k$, $\displaystyle \lVert V_n^jV_n^k\rVert_{L_t^{q/2}L_x^{r/2}}\mathop{\longrightarrow}_{n\rightarrow\infty} 0$.  Moreover, note that for every $j\geq 1$, ($\ref{lab22}$) and ($\ref{lab9}$) give:
\begin{align*}
\limsup_{n\rightarrow\infty}\lVert V_n^j\rVert_{L_t^qL_x^r}&=\limsup_{n\rightarrow\infty}\lVert V^j\rVert_{L_t^qL_x^r}\\
&\leq C\limsup_{n\rightarrow\infty}\left[E((V_n^k(0),\partial_t V_n^k(0)))\right]^\frac{1}{2}\\
&\leq C\limsup_{n\rightarrow\infty} \left[E((u_{0,n},u_{1,n}))\right]^\frac{1}{2}.
\end{align*}
Thus, $(u_{0,n},u_{1,n})$ bounded in $\dot{H}^1\times L^2$ implies that each of the terms $(Ia)_{j,k}, (Ib)_{j,k_1,\cdots,k_s}, (Ic)_{j,i},$ and $(II)_j$ tends to $0$ as $n\rightarrow\infty$.  This completes the argument.
\end{proof}
We now focus on the proof of the main theorem in this section, which was motivated by the work of Theorem $1.2$ in \cite{ShaoSchrodinger}.

\begin{proof}[Proof of Theorem \ref{lab12}]
From the definition of the supremum, choose $(\phi_n,\psi_n)\subset\dot{H}^1\times L^2$ such that $\lVert (\phi_n, \psi_n)\rVert_{\dot{H}^1\times L^2(\mathbb{R}^d)}=1$ and $\lVert S(t)(\phi_n,\psi_n)\rVert_{L_t^qL_x^r}{\displaystyle \mathop{\longrightarrow}_{n\rightarrow\infty}} W_{q,r}$.  Then using the profile decomposition, Theorem $\ref{lab4}$, there exists a subsequence $(\phi_n,\psi_n)$, triples $(\epsilon_n^j,x_n^j,t_n^j)$, a sequence $(V_0^j,V_1^j)$, $V^j$ and $V_n^j$.

Let $\epsilon>0$ be given.  Using our choice of $(\phi_n,\psi_n)$ and Corollary $\ref{lab85}$, we choose $N_0\in\mathbb{N}$ such that for every $n,l\geq N_0$,
\begin{align}
\label{lab94}
W_{q,r}-\lVert S(t)(\phi_n,\psi_n)\rVert_{L_t^qL_x^r}&<\frac{\epsilon}{2},
\end{align}
and
\begin{align*}
\limsup_{n\rightarrow\infty} \lVert S(t)(w_{0,n}^l,w_{1,n}^l)\rVert_{L_t^qL_x^r}<\frac{\epsilon}{2}.
\end{align*}
Then, applying the definition of limsup, there exists $N_1>N_0$ such that for every $n,l\geq N_1$,
\begin{align}
\label{lab95}
\lVert S(t)(w_{0,n}^l,w_{1,n}^l)\rVert_{L_t^qL_x^r}&\leq\frac{\epsilon}{2},
\end{align}
and if $B=\min\{q,r\}$,
\begin{align}
\label{lab96}
\lVert \sum_{j=1}^l V_n^j\rVert_{L_t^qL_x^r}^B&\leq\sum_{j=1}^l \lVert V^j\rVert_{L_t^qL_x^r}^B+\epsilon,
\end{align}
where we have used Lemma $\ref{lab86}$ for ($\ref{lab96}$).

Let $n,l>N_1$ be given.  Then by applying ($\ref{lab96}$) followed by ($\ref{lab6}$)-($\ref{lab7}$) we obtain
\begin{align}
\nonumber \sum_{j=1}^l \lVert V^j\rVert_{L_t^qL_x^r}^B&\geq\lVert \sum_{j=1}^l V_n^j\rVert_{L_t^qL_x^r}^B-\epsilon\\
\nonumber &=\lVert S(t)(\phi_n,\psi_n)-S(t)(w_{0,n}^l,w_{1,n}^l)\rVert_{L_t^qL_x^r}^B-\epsilon\\
\nonumber &\geq\left(\lVert S(t)(\phi_n,\psi_n)\rVert_{L_t^qL_x^r}-\lVert S(t)(w_{0,n}^l,w_{1,n}^l)\rVert_{L_t^qL_x^r}\right)^B-\epsilon\\
\label{lab97} &\geq\left(W_{q,r}-\epsilon\right)^B-\epsilon,
\end{align}
where ($\ref{lab97}$) follows from ($\ref{lab94}$)-($\ref{lab95}$).

Choose $j_0(l)\in \{1,\cdots, l\}$ such that $\lVert V^{j_0(l)}\rVert_{L_t^qL_x^r}=\max \{\lVert V^j\rVert_{L_t^qL_x^r}:j=1,\cdots,l\}$ and $j_0(l)$ is the smallest integer for which this holds.  Then applying the Strichartz inequality we see that,
\begin{align}
\nonumber (W_{q,r}-\epsilon)^B-\epsilon&\leq\sum_{j=1}^l \lVert V^j\rVert_{L_t^qL_x^r}^B\leq\lVert V^{j_0(l)}\rVert_{L_t^qL_x^r}^{B-2}\left(\sum_{j=1}^l \lVert V^j\rVert_{L_t^qL_x^r}^2\right)\\
\nonumber &\leq W_{q,r}^{B-2}\left(\lVert V_0^{j_0(l)}\rVert_{\dot{H}^1}^2+\lVert V_1^{j_0(l)}\rVert_{L^2}^2\right)^\frac{B-2}{2}\left(\sum_{j=1}^l W_{q,r}^2\left(\lVert V_0^j\rVert_{\dot{H}^1}^2+\lVert V_1^j\rVert_{L^2}^2\right)\right)\\
\label{lab98}&\leq W^B_{q,r}\left(\lVert V_0^{j_0(l)}\rVert_{\dot{H}^1}^2+\lVert V_1^{j_0(l)}\rVert_{L^2}^2\right)^\frac{B-2}{2}\leq W^B_{q,r},
\end{align}
where we twice use the fact \begin{align}
\sum_{j=1}^l \left(\lVert V_0^j\rVert_{\dot{H}^1}^2+\lVert V_1^j\rVert_{L^2}^2\right)\leq\lim_{n\rightarrow\infty}\lVert \phi_n\rVert_{\dot{H}^1}^2+\lVert \psi_n\rVert_{L^2}^2+o(1)=1.
\label{lab100}
\end{align}
which follows from the identity ($\ref{lab9}$).

We want to consider the limit as $\epsilon$ tends to $0$ in the above chain of inequalities ($\ref{lab98}$).  However,
our choice of $j_0(l)$ depends on $l$, which in turn depends on $\epsilon$.
In Appendix B, we provide an argument to show that we may obtain ($\ref{lab98}$), with $j_0(l)$ replaced
by an index independent of $\epsilon$ for $\epsilon$ small enough.

Let $\epsilon_0$ be as selected in Appendix B.  Then for $0<\epsilon<\epsilon_0$, we have
\begin{align*}
(W_{q,r}-\epsilon)^B-\epsilon\leq W_{q,r}^B\left(\lVert V_0^{j_0(M)}\rVert_{\dot{H}^1}^2+\lVert V_1^{j_0(M)}\rVert_{L^2}^2\right)^\frac{B-2}{2}\leq W_{q,r}^B.
\end{align*}

Letting $\epsilon\rightarrow 0$, we see that $\left(\lVert V_0^{j_0(M)}\rVert_{\dot{H}^1}^2+\lVert V_1^{j_0(M)}\rVert_{L^2}^2\right)^\frac{B-2}{2}=1$, so that \begin{align*}
\lVert V_0^{j_0(M)}\rVert_{\dot{H}^1}^2+\lVert V_1^{j_0(M)}\rVert_{L^2}^2=1.
\end{align*}

Then ($\ref{lab100}$) implies that for all $j\neq j_0(M)$, $\lVert V_0^j\rVert_{\dot{H}^1}^2+\lVert V_1^j\rVert_{L^2}^2=0$ which, combined with ($\ref{lab97}$), gives that for $0<\epsilon<\epsilon_0$,
\begin{align*}
(W_{q,r}-\epsilon)^B-\epsilon\leq\sum_{j=1}^l \lVert V^j\rVert_{L_t^qL_x^r}^B=\lVert V^{j_0(M)}\rVert_{L_t^qL_x^r}^B.
\end{align*}

Taking $\epsilon\rightarrow 0$ shows that $W_{q,r}^B\leq\lVert V^{j_0(M)}\rVert_{L_t^qL_x^r}^B$, and thus, taking power $1/B$ and using the Strichartz inequality,
\begin{align*}
\lVert V^{j_0(M)}\rVert_{L_t^qL_x^r}=W_{q,r}=W_{q,r}(\lVert V_0^{j_0(M)}\rVert_{\dot{H}^1}^2+\lVert V_0^{j_0(M)}\rVert_{L^2}^2)^\frac{1}{2},
\end{align*}
and hence $(V_0^{j_0(M)},V_1^{j_0(M)})$ is a maximizing pair as desired.
\end{proof}
\section{MAXIMIZERS FOR THE $\dot{H}^s\times\dot{H}^{s-1}$-STRICHARTZ INEQUALITIES}

We now turn our attention to the inequalities ($\ref{lab13}$).  As in the previous section, we use a linear profile decomposition to prove the existence of maximizers.  In this setting, we obtain the relevant profile decomposition Theorem $\ref{lab14}$ as a consequence of Theorem $\ref{lab4}$ applied to an appropriate sequence of initial data.  Our result on the existence of maximizers then follows using an argument identical to that given in the preceding section.

\subsection{Proof of Theorem $\ref{lab14}$}\ \\
\ \\
We follow the approach in \cite{ShaoSchrodinger}.  Let $s\geq 1$ be given and let $(u_{0,n},u_{1,n})_{n\in\mathbb{N}}$ be a bounded sequence in $\dot{H}^s\times \dot{H}^{s-1}(\mathbb{R}^d)$.  For
each $n\in\mathbb{N}$, define \begin{align*}\label{lab101}f_{0,n}:= D^{s-1}u_{0,n},\quad f_{1,n}:=D^{s-1}u_{1,n},\end{align*} so that $(f_{0,n},f_{1,n})$ is a bounded
sequence in $\dot{H}^1\times L^2$.
We can then apply Theorem $\ref{lab4}$ to obtain a subsequence $(f_{0,n},f_{1,n})$, triples $(\epsilon_n^j,x_n^j,t_n^j)\in \mathbb{R}^+\times\mathbb{R}^d\times\mathbb{R}$ and a sequence $(\psi_0^j,\psi_1^j)\subset\dot{H}^1\times L^2$.  Define $\psi^j$ and $\psi_n^j$ as in the statement of the theorem.  In particular, we have for $t\in\mathbb{R}, x\in\mathbb{R}^d$, \begin{align*}S(t)(f_{0,n},f_{1,n})(x)&=\sum_{j=1}^l \psi_n^j(t,x)+S(t)(r_{0,n}^l,r_{1,n}^l)(x).\end{align*}
Let $(u_{0,n},u_{1,n})$ be the subsequence corresponding to $(f_{0,n},f_{1,n})$.  Without loss of generality, assume all $f_{0,n},f_{1,n},\psi_0^j,\psi_1^j$ are Schwartz.  Taking $D^{1-s}$ of both sides, we get
\begin{align*}
\lefteqn{S(t)(u_{0,n},u_{1,n})}\\
&=D^{1-s}(S(t)(f_{0,n},f_{1,n}))=D^{1-s}(S(t)(D^{s-1}u_{0,n},D^{s-1}u_{1,n}))\\
&=\sum_{j=1}^l D^{1-s}\left[\frac{1}{(\epsilon_n^j)^\frac{d-2}{2}}\psi^j(\frac{t-t_n^j}{\epsilon_n^j},\frac{x-x_n^j}{\epsilon_n^j})\right](t,x)+D^{1-s}[S(t)(r_{0,n}^l,r_{1,n}^l)(x)]\\
&=\sum_{j=1}^l \frac{1}{(\epsilon_n^j)^{\frac{d-2}{2}-(s-1)}}(D^{1-s}\psi^j)\left(\frac{t-t_n^j}{\epsilon_n^j},\frac{x-x_n^j}{\epsilon_n^j}\right)+S(t)(D^{1-s}r_{0,n}^l,D^{1-s}r_{1,n}^l)(x).
\end{align*}

Note that $(\partial_{tt}-\Delta) \psi^j=0$ implies $(\partial_{tt}-\Delta)(D^{1-s}\psi^j)=0$ and we set
\begin{align*}
V_0^j=(D^{1-s}\psi^j)(0,x),\quad V_1^j=(\partial_t D^{1-s}\psi^j)(0,x)
\end{align*}
so that $V^j(t,x):=S(t)(V_0^j,V_1^j)(x)=D^{1-s}\psi^j(t,x)$.  Then, if
$V_n^j:=\frac{1}{(\epsilon_n^j)^{\frac{d-2}{2}-(s-1)}}V^j\left(\frac{t-t_n^j}{\epsilon_n^j},\frac{x-x_n^j}{\epsilon_n^j}\right)$ and \begin{align*}w_{0,n}^l:= D^{1-s}r_{0,n}^l,\quad w_{1,n}^l:=D^{1-s}r_{1,n}^l, \end{align*} we have
\begin{align*}
S(t)(u_{0,n},u_{1,n})=\sum_{j=1}^l V_n^j(t,x)+S(t)(w_{0,n}^l,w_{1,n}^l)(x),
\end{align*}
which gives ($\ref{lab15}$)-($\ref{lab16}$).

For $(q,r)$ satisfying the $\dot{H}^s$-scaling condition, Sobolev's inequality implies that
\begin{align*}
\lVert S(t)(w_{0,n}^l,w_{1,n}^l)(x)\rVert_{L_t^qL_x^r}&\leq C\lVert D^{s-1}S(t)(w_{0,n}^l,w_{1,n}^l)(x)\rVert_{L_t^qL_x^\frac{rd}{d+(s-1)r}}.
\end{align*}

Then $(q,\frac{rd}{d+(s-1)r})$ is a wave admissible pair with $q,r\in (2,\infty)$ and satisfying the $\dot{H}^1$-scaling condition, so that Corollary $\ref{lab85}$ gives that \begin{align*}\lVert D^{s-1}S(t)(w_{0,n}^l,w_{1,n}^l)(x)\rVert_{L_t^qL_x^\frac{rd}{d+(s-1)r}}\mathop{\longrightarrow}_{n\rightarrow\infty} 0.\end{align*} This in turn implies $\lVert S(t)(w_{0,n}^l,w_{1,n}^l)(x)\rVert_{L_t^qL_x^r}\displaystyle{\mathop{\longrightarrow}_{n\rightarrow\infty}} 0$, which gives ($\ref{lab17}$).

To complete the theorem, it remains to verify the limit ($\ref{lab18}$).  From the profile decomposition,
\begin{align*}
\lVert f_{0,n}\rVert_{\dot{H}^1}^2+\lVert f_{1,n}\rVert_{L^2}^2&=\sum_{j=1}^l \left(\lVert \psi_0^j\rVert_{\dot{H}^1}^2+\lVert \psi_1^j\rVert_{L^2}^2\right)+\lVert r_{0,n}^l\rVert_{\dot{H}^1}^2+\lVert r_{1,n}^l\rVert_{L^2}^2+o(1),
\end{align*}
so that
\begin{align*}
\lefteqn{\lVert u_{0,n}\rVert_{\dot{H}^s}^2+\lVert u_{1,n}\rVert_{\dot{H}^{s-1}}^2}&\\
&\hspace{0.75in}=\lVert f_{0,n}\rVert_{\dot{H}^1}^2+\lVert f_{1,n}\rVert_{L^2}^2\\
&\hspace{0.75in}=\sum_{j=1}^l \left(\lVert D^{1-s}\psi_0^j\rVert_{\dot{H}^s}^2+\lVert D^{1-s}\psi_1^j\rVert_{\dot{H}^{s-1}}^2\right)+\lVert D^{1-s}r_{0,n}^l\rVert_{\dot{H}^s}^2+\lVert D^{1-s}r_{1,n}^l\rVert_{\dot{H}^{s-1}}^2+o(1)\\
&\hspace{0.75in}=\sum_{j=1}^l \left(\lVert V_0^j\rVert_{\dot{H}^s}^2+\lVert V_1^j\rVert_{\dot{H}^{s-1}}^2\right)+\lVert w_{0,n}^l\rVert_{\dot{H}^s}^2+\lVert w_{1,n}^l\rVert_{\dot{H}^{s-1}}^2+o(1).
\end{align*}
The proof of Theorem $\ref{lab14}$ is now complete.
\subsection{Existence of Maximizers}\ \\
\ \\
We now arrive to the proof of Theorem $\ref{lab20}$.  Let us note that an argument similar to the proof of Lemma $\ref{lab31}$ shows that the analogue of ($\ref{lab11}$) holds in the setting of Theorem $\ref{lab14}$.  More precisely, if the hypotheses of Theorem $\ref{lab14}$ are satisfied, then for every $j\neq k$, \begin{align*}\lVert V_n^jV_n^k\rVert_{L_{t,x}^\frac{d+1}{d-2s}}\mathop{\longrightarrow}_{n\rightarrow\infty} 0.\end{align*}  We can then use this result to prove the analogue of Lemma $\ref{lab86}$ for the $\dot{H}^s$-scaling condition.  The proof of Theorem $\ref{lab20}$ then proceeds almost identically to that of Theorem $\ref{lab12}$, using the Strichartz inequality ($\ref{lab23}$) for the space $\dot{H}^s\times \dot{H}^{s-1}$ in place of $\dot{H}^1\times L^2$.

\section{APPENDIX A}
In this appendix we give the proof of Lemma $\ref{lab25}$, adapting the case $d=3$ given in \cite{BahouriGerard} to higher dimensions.
\begin{proof}[Proof of Lemma $\ref{lab25}$] \
For every $A>0$, we define $u_{<A},u_{>A}$ by
\[ \hat{u}_{<A}(\xi)=\chi_{|\xi|\leq A} \hat{u}(\xi), \quad \hat{u}_{<A}(\xi)=\chi_{|\xi|> A} \hat{u}(\xi). \]
We have
\begin{align}
\lVert u_{<A}\rVert_{L^\infty} &\leq (2\pi)^{-\frac{d}{2}}\lVert\hat{u}_{<A}\rVert_{L^1}\nonumber\\
                      &= (2\pi)^{-\frac{d}{2}}\lVert\chi_{|\xi|\leq A}\hat{u}(\xi)\rVert_{L^1}\nonumber\\
                      &\leq (2\pi)^{-\frac{d}{2}}\sum_{k\in \mathbb{Z},k\leq k(A)}\lVert\chi_{2^k\leq |\xi| \leq 2^{k+1}}\hat{u}(\xi)\rVert_{L^1}\label{lab102}
\end{align}
where $k(A)$ is the largest integer such that $2^{k(A)}\leq A$.\\
Now, Schwarz's inequality gives,
\begin{align*}
\lVert\chi_{2^{k} \leq \xi\leq 2^{k+1}} \hat{u}\rVert_{L^{1}}& = \int|\chi_{2^k \leq \xi\leq 2^{k+1}}\frac{\xi}{\xi}\hat{u}(\xi)| d\xi\\
                                                 &\leq(\int|\frac{1}{|\xi|^{2}}\chi_{2^k \leq \xi\leq 2^{k+1}} d\xi)^\frac{1}{2}
                                                       (\int\chi_{2^k \leq \xi\leq 2^{k+1}}|\xi|^{2}|\hat{u}(\xi)| d\xi)^\frac{1}{2}\\
                                                 &\leq ( \frac{1}{2^k}2^\frac{dk}{2}) I_k (\nabla u)\\
                                                 &\leq ( \frac{1}{2^k}2^\frac{dk}{2}) \lVert\nabla u\rVert_B.
\end{align*}
Combining this with ($\ref{lab102}$), we obtain
\begin{align}
\lVert u_{<A}\rVert_{L^\infty}&\leq KA^{1/2}\lVert \nabla u\rVert_B,\label{lab103}
\end{align}
where $K$ is some positive constant.\\
We also have,
\begin{align}
\lVert u\rVert^p_{L^p}=p\int^\infty_0 \lambda^{p-1} m\{|u|>\lambda\}d\lambda\label{lab104}
\end{align}
where $m$ is the Lebesgue measure on $\mathbb{R}^d$ and we set
\[ A(\lambda)=(\frac{\lambda}{2K \lVert\nabla u\rVert_B})^\frac{2}{d-2}.\]
Then by ($\ref{lab103}$)
\begin{align*}
\lVert u_{<A(\lambda)}\rVert_{L^{\infty}}\leq K(A(\lambda))^{\frac{d-2}{2}}\lVert \nabla u\rVert_B =  \frac{\lambda}{2}
\end{align*}
and hence,
\begin{align*}
  m\{|u|>\lambda\}\leq m\{u_{>A(\lambda)}>\frac{\lambda}{2}\}\leq \frac{4}{\lambda^2}\lVert u_{>A(\lambda)}\rVert_{L^2(\mathbb{R}^d)}^2=\frac{4}{\lambda^2}\lVert \hat u_{>A(\lambda)}\rVert_{L^2(\mathbb{R}^d)}^2
\end{align*}
Coming back to ($\ref{lab104}$) and plugging in this estimate, we have,
\begin{align*}
\lVert u\rVert^p_{L^p}&= p\int^\infty_0 \lambda^{p-1} m\{|u|>\lambda\}d\lambda\\
&\leq 4p\int^{\infty}_0 \lambda^{p-3}\left(\int_{|\xi|>A(\lambda)} |\hat u(\xi)|^2d\xi\right)d\lambda\\
&= 4p \frac{\left(2K\lVert \nabla u\rVert_B\right)^{p-2}}{p-2}\int_{\mathbb{R}^d} |\xi|^2 |\hat u(\xi)|^2 d\xi\\
&= C\lVert\nabla u\rVert_{L^2(\mathbb{R}^d)}^2\lVert \nabla u\rVert_B^{p-2},
\end{align*}
where we use Fubini's theorem in the second inequality.  Hence, we obtain the desired inequality ($\ref{lab26}$).
\end{proof}

\section{APPENDIX B}
In this appendix, we provide an argument to remove the dependence of $j_0(l)$ on $\epsilon$ in the chain of inequalities ($\ref{lab98}$).

Note that taking $l\rightarrow\infty$ in ($\ref{lab100}$)
shows that \begin{align*}\lVert V_0^j\rVert_{\dot{H}^1}^2+\lVert V_1^j\rVert_{L^2}^2\mathop{\rightarrow}_{j\rightarrow\infty} 0,\end{align*} so
that we can choose $M\in\mathbb{N}$ such that for all $j\geq M$,
\begin{align*}\lVert V_0^j\rVert_{\dot{H}^1}^2+\lVert V_1^j\rVert_{L^2}^2<\left(\frac{1}{2}\right)^\frac{2}{B-2}.\end{align*}

Let $\epsilon_0>0$ be small enough such that $(W_{q,r}-\epsilon_0)^B-\epsilon_0>\frac{W_{q,r}^B}{2}>0$.  Then for each $0<\epsilon<\epsilon_0$, \begin{align}
\label{lab105}
(W_{q,r}-\epsilon)^B-\epsilon>\frac{W_{q,r}^B}{2}.
\end{align}

For each $0<\epsilon<\epsilon_0$, applying the argument at the beginning of the proof of Theorem $\ref{lab12}$ with $\epsilon$ to find $N_1\in\mathbb{N}$ as stated, we claim that if $l>\max\{N_1,M\}$, then $j_0(l)=j_0(M)$.  To see this, let $l>\max\{N_1,M\}$ be given.  Then, combining ($\ref{lab105}$) with ($\ref{lab98}$), we obtain
\begin{align*}
\frac{1}{2}<\frac{(W_{q,r}-\epsilon)^B-\epsilon}{W_{q,r}^B}\leq (\lVert V_0^{j_0(l)}\rVert_{\dot{H}^1}^2+\lVert V_1^{j_0(l)}\rVert_{L^2}^2)^\frac{B-2}{2}
\end{align*}
so that the choice of $M$ implies $j_0(l)<M$, and thus $\lVert V^{j_0(M)}\rVert_{L_t^qL_x^r}=\max \{\lVert V^j\rVert:i=1,\cdots,M\}\geq \lVert V^{j_0(l)}\rVert_{L_t^qL_x^r}$.  But we have $\lVert V^{j_0(M)}\rVert_{L_t^qL_x^r}\leq \lVert V^{j_0(l)}\rVert_{L_t^qL_x^r}$ so that $\lVert V^{j_0(M)}\rVert_{L_t^qL_x^r}=\lVert V^{j_0(l)}\rVert_{L_t^qL_x^r}=\max\{ \lVert V^j\rVert_{L_t^qL_x^r}:j=1,\cdots,l\}$.  Then $j_0(l)$ the smallest integer such that $\lVert V^{j_0(l)}\rVert=\max\{ \lVert V^j\rVert:j=1,\cdots,l\}$ implies $j_0(l)\leq j_0(M)$.  Note also that $M\leq l$ implies $j_0(M)\leq j_0(l)$ and therefore, $j_0(l)=j_0(M)$.

\subsection*{Acknowledgements} I am deeply grateful to my advisors William Beckner and Nata\v{s}a Pavlovi\'{c} for all of their generous support and encouragement during the preparation of this work.  I am also indebted to Nata\v{s}a Pavlovi\'{c} for suggesting the problem and for many useful comments that helped to significantly improve the manuscript.  I would also like to thank Magdalena Czubak for useful discussions related to the profile decomposition and Shuanglin Shao for helpful comments.


\end{document}